\documentclass[12pt, oneside]{amsart}   	
\usepackage{geometry}
\numberwithin{equation}{section}                		
\numberwithin{equation}{section}
\usepackage{amsmath,amssymb,amsthm,amscd}	
\usepackage{geometry}                		
\usepackage{graphicx}	
\usepackage{ascmac}
\usepackage{indentfirst}
\usepackage{longtable}
\usepackage{cases}
\usepackage[all]{xy}
\usepackage[frenchb, english]{babel}
\usepackage[OT2,T1]{fontenc}

\DeclareSymbolFont{cyrletters}{OT2}{wncyr}{m}{n}
\DeclareMathSymbol{\Sha}{\mathalpha}{cyrletters}{"58}
\DeclareMathOperator{\Gal}{\rm{Gal}}

\newcommand{\Q}[1][]{\mathbb{Q}}
\newcommand{\Z}[1][]{\mathbb{Z}}
\newcommand{\F}[1][]{\mathbb{F}}
\newcommand{\C}[1][]{\mathbb{C}}
\newcommand{\Frob}[1][]{\mathrm{Frob}}
\newcommand{\ch}[1][]{\chi_{\cyc}}
\newcommand{\cyc}[1][]{\mathrm{cyc}}
\newcommand{\Sel}[1][]{\mathrm{Sel}}
\newcommand{\Ker}[1][]{\mathrm{Ker}}
\newcommand{\Res}[1][]{\mathrm{Res}}
\newcommand{\Hom}[1][]{\mathrm{Hom}}
\newcommand{\cl}[1][]{\mathrm{Cl}}
\newcommand{\ur}[1][]{\mathrm{ur}}
\newcommand{\A}[1][]{\mathcal{A}}
\newcommand{\BK}[1][]{\mathrm{BK}}

\newtheorem{thm}{Theorem}[section]
\newtheorem*{thm*}{Theorem}
\newtheorem{lem}[thm]{Lemma}
\newtheorem{prop}[thm]{Proposition}

\newtheorem{defn}[thm]{Definition}
\newtheorem{rem}[thm]{Remark}
\newtheorem{que}[thm]{Question}

\newcommand{\GL}{\mathrm{GL}}

\title[]{Ideal class groups of number fields associated to modular Galois representations}

\author{Naoto Dainobu}



\address{Department of Mathematics \\
3-14-1 Hiyoshi, Kohoku-ku, Yokohama-shi, Kanagawa 223-8522 Japan}
\email{vicarious@keio.jp}

\begin{document}
\maketitle
\begin{abstract}
Let $p$ be an odd prime number and $f$ a modular form. We consider the $\F_p$-valued Galois representation $\bar{\rho}_f$ attached to $f$ and its twist $\bar{\rho}_{f, D}$ by the quadratic character $\chi_D$ corresponding to a quadratic discriminant $D$. We define $K_{f, D}$ to be the field corresponding to the kernel of $\bar{\rho}_{f, D}$. In this article, we investigate the ideal class group $\cl(K_{f, D})$ of the number field $K_{f, D}$ as a $\Gal(K_{f, D}/\Q)$-module. We give a condition which implies the existence of a $\Gal(K_{f, D}/\Q)$-equivariant surjective homomorphism from $\cl(K_{f, D})\otimes \F_p$ to the representation space $M_{f, D}$ of $\bar{\rho}_{f, D}$, using Bloch and Kato's Selmer group of $\bar{\rho}_{f, D}$. We also give some numerical examples where we have such surjections by calculating the central value of the $L$-function of $f$ twisted by $\chi_D$ under Bloch and Kato's conjecture. Our main result in this paper is a partial generalization of the previous result of Prasad and Shekhar on elliptic curves to higher weight modular forms.
\end{abstract}
\section{Introduction and main result}
Galois representations attached to modular forms are used as fundamental tools in several places of modern number theory. For example, they play important roles in the proof of the Iwasawa main conjectures by Mazur-Wiles \cite{BA} and by Skinner-Urban \cite{CE}. The associated number fields to such Galois representations are also important objects in number theory. For instance, in \cite[Section 6.4]{EC}, Bosman computed the images of Galois representations attached to modular forms explicitly, investigating their associated number fields.

In this article, our interest is in the ideal class groups of such important number fields. Let us explain slightly more details here. Let $p$ be an odd prime number and $f$ a modular form with certain conditions we explain later. Then we can attach an $\F_p$-valued Galois representation $\bar{\rho}_f : G_{\Q} \rightarrow  \GL_2(\F_p)$, where $G_{\Q}$ denotes the absolute Galois group $\Gal(\overline{\Q}/\Q)$ for a fixed algebraic closure $\overline{\Q}$ of $\Q$. We also consider its twist $$\bar{\rho}_{f, D} : G_{\Q} \rightarrow \mathrm{Aut}_{\F_p}(M_{f, D})\simeq \GL_2(\F_p)$$ by the  quadratic character $\chi_D$ corresponding to a quadratic discriminant $D$. Here, we define $M_{f, D}$ to be the representation space of $\bar{\rho}_{f, D}$. Then we can define a number field $K_{f, D}$ as the fixed field of $\Ker(\bar{\rho}_{f, D})$ in $\overline{\Q}$. We study the ideal class group $\cl(K_{f, D})$ of $K_{f, D}$. In this situation, the Galois group $\Gal(K_{f, D}/\Q)$ acts naturally on $\cl(K_{f, D})$, so we investigate it as a $\Gal(K_{f, D}/\Q)$-module. More precisely, we are interested in the following question.
\begin{que}\label{?}
For an arbitrary $\F_p$-valued irreducible representation $M$ of $\Gal(K_{f, D}/\Q)$, when does a $\Gal(K_{f, D}/\Q)$-equivariant surjective homomorphism from $\cl(K_{f, D})\otimes\F_p$ to $M$ exist? 
\end{que}
This question is important for us to understand the structure of $\cl(K_{f, D})$ as a $\Gal(K_{f, D}/\Q)$-module. In fact, the existence of such a surjection implies that $M$ appears in the semi-simplification $(\cl(K_{f, D})\otimes\F_p)^{\mathrm{ss}}$ of $\cl(K_{f, D})\otimes\F_p$. The author hopes that Question \ref{?} will lead us to a non-commutative analogue of Herbrand and Ribet's theorem in \cite{He, Ri} for $K_{f, D}/\Q$.

As the first step toward Question \ref{?}, we deal with the case $M=M_{f, D}$ in this paper. Our main result gives an answer to Question \ref{?} for $M=M_{f, D}$ in terms of Bloch and Kato's Selmer group $H^1_f(\Q, M_{f, D})$ of $M_{f, D}$. 
\begin{thm*}[Main result]\label{0-main}
Under the conditions {\rm (A)}, {\rm (B)}, {\rm (C)} and {\rm (D)} in Subsection {\rm 2.3}, there exists a $\Gal(K_{f, D}/\Q)$-equivariant surjective homomorphism from $\cl(K_{f, D})\otimes \F_p$ to $M_{f, D}$ if the inequality $\mathrm{dim}_{\F_p}( H_f^1(\Q, M_{f, D})) \geqslant 2$ holds.
\end{thm*}

We describe our main result in a precise form as Theorem \ref{main} in Subsection 2.3 again. \\

Using our main result, we can study the ideal class group $\cl(K_{f, D})$ via the size of the Bloch-Kato Selmer group $H_f^1(\Q, M_{f, D})$. Assuming Bloch and Kato's conjecture, we can examine when the inequality $\mathrm{dim}_{\F_p}( H_f^1(\Q, M_{f, D})) \geqslant 2$ holds by calculating a special value of the $L$-function $L(f, \chi_D, s)$ of $f$ twisted by $\chi_D$. In Section 6, we do such calculations and give some examples of our result (see Propositions \ref{exam1} and \ref{exam2}). 

Since $K_{f, D}$ is actually a non-abelian extension over $\Q$, it is difficult in general to study $\cl(K_{f, D})$ as a $\Gal(K_{f, D}/\Q)$-module directly even with computer algebra systems. Therefore, the author hopes our result, which clarifies a part of the structure of $\cl(K_{f, D})$ as a $\Gal(K_{f, D}/\Q)$-module, will be an important tool in the topic in this paper.

Here, we give several known results on the ideal class group $\cl(K_{f, D})$. Suppose $f$ is a rational cusp form of weight $2$.  Then the Galois representation $M_{f, D}$ is isomorphic to the group of the $p$-torsion points $E[p]$ of $E$, where $E$ is the quadratic twist of the elliptic curve corresponding to $f$ by $D$. Hence, the number field $K_{f, D}$ is the same as $\Q(E[p])$ which is generated by the coordinates of the points in $E[p]$ over $\Q$. In \cite{PS}, Prasad and Shekhar studied the ideal class group $\cl(\Q(E[p]))$ as a $\Gal(\Q(E[p])/\Q)$-module using the $p$-Selmer group $\Sel(\Q, E[p])$ of $E$. One of their results \cite[Theorem 3.1]{PS} is that if $\mathrm{dim}_{\F_p}(\Sel(\Q, E[p]))\geqslant 2$, then there exists a $\Gal(\Q(E[p])/\Q)$-equivariant surjection $\cl(\Q(E[p]))\otimes \F_p \twoheadrightarrow E[p]$ under similar assumptions to ours. Actually, the motivation of our work in this paper was to generalize this result of Prasad and Shekhar to higher weight modular forms.

The order of $\cl(K_{f, D})$ has been also studied by many people. For example, some lower bounds for the $p$-adic valuation of $\#\cl(\Q(E[p]))$ were given by Sairaiji-Yamauchi \cite{SY} and Hiranouchi \cite{H} for an elliptic curve $E$. Ohshita \cite{Oh} further generalized their result for number fields associated to  more general Galois representations including our $\bar{\rho}_{f, D}$, and gave a lower bound in \cite[Theorem 3.4]{Oh} similar to \cite{SY}, \cite{H}. On the other hand, we can also give a lower bound for the $p$-adic valuation of $\#\cl(K_{f, D})$ from our main result. In fact, if we have a surjective homomorphism $$\cl(K_{f, D})\otimes \F_p \twoheadrightarrow M_{f, D},$$ then we obtain especially $p^2\mid \#\cl(K_{f, D})$ since $M_{f, D}$ is $2$-dimensional over $\F_p$. This lower bound is the same as Ohshita's given in \cite{Oh}. 
\begin{rem}\label{sym}
In our previous work \cite{D}, we generalized the result of Prasad and Shekhar \cite{PS} to a different direction from that of this paper. For an elliptic curve $E$ over $\Q$, we gave in \cite{D} a condition which implies the existence of a surjective $\Gal(\Q(E[p])/\Q)$-equivariant homomorphism $$\cl(\Q(E[p]))\otimes\F_p\twoheadrightarrow \mathrm{Sym}^i E[p] 
\quad (1\leqslant i \leqslant p-2).$$ Here $\mathrm{Sym}^i E[p]$ denotes the $i$-th symmetric power of $E[p]$ and $i=1$ is the case Prasad and Shekhar studied. Thus, combining the method of this paper and \cite{D}, we can give a condition which implies the existence of a surjective $\Gal(K_{f, D}/\Q)$-equivariant homomorphism $$\cl(K_{f, D})\otimes\F_p\twoheadrightarrow \mathrm{Sym}^i M_{f, D},$$ and an answer to Question \ref{?} for $M=\mathrm{Sym}^i M_{f, D}$ (see Remark \ref{symst}).
\end{rem}

At the end of this section, we explain the outline of this paper. In Section 2, we first introduce notation and recall the definition of the Bloch-Kato Selmer groups and Tate-Shafarevich groups for $p$-adic representations. After that, we state our main result in Theorem \ref{main} again in a precise form and write a sketch of the proof dividing it into Step 1, Step 2 and Step 3. In Section 3, we prove Step 1. In Section 4, we introduce basic notions of the Tamagawa factor and prove Step 2. In Section 5, we prove Step 3 and complete the proof of Theorem \ref{main}. Finally in Section 6, we introduce some numerical examples.

\section{Preliminaries and a sketch of the proof of the main result}
\subsection{Basic notation}\label{notation}
First, we summarize our notations for various Galois representations we use in this paper. 

We fix $p$ as an odd prime number. Let $f(z)=\sum_{n=0}^{\infty} a_n q^n$ be a normalized Hecke eigen new cuspform of even weight $k\geqslant 2$ and level $\Gamma_0(N)$. Here $q:=e^{2\pi\sqrt{-1}z}$ and the parameter $z$ is in the complex upper half plane. We define a number field $\Q(f):=\Q(a_1, a_2, \cdots)$ generated by the Fourier coefficients $\{a_1, a_2, \cdots\}$ of $f$ over $\Q$. We assume the following two conditions hold.
\begin{itemize}
\item There exists a prime ideal $\mathfrak{p}$ of $\Q(f)$ above $p$ such that the completion of $\Q(f)$ at $\mathfrak{p}$ is isomorphic to $\Q_p$.
\item $\Q(f)$ is totally real.
\end{itemize}
By the first assumption and the result of Eichler \cite{Ei}, Shimura \cite{Sh} and Deligne \cite{De}, we can take the $p$-adic Galois representation $\rho^0_f : G_{\Q} \rightarrow \mathrm{Aut}_{\Q_p}(V^0_f)$ attached to $f$ and $\mathfrak{p}$, where $V^0_f$ denotes its representation space. Note that we take $V^0_{f}$ whose Hodge-Tate weights are $\{0, k-1\}$ as a $G_{\Q_p}$-module. For a fixed Galois stable $\mathbb{Z}_p$-lattice $T^0_f$ of $V^0_f$, we put 
\[
A^0_f := V^0_f/T^0_f\cong (\Q_p/\Z_p)^{\oplus 2},\ \ M^0_f:=T^0_f/pT^0_f\cong \F_p^{\oplus 2}
\]
 and $$\bar{\rho}^0_f:G_{\Q} \rightarrow \mathrm{Aut}_{\F_p}(M^0_f)\cong\mathrm{GL}_2(\F_p).$$ 
 
 Let $\chi_{\cyc} : G_{\Q} \rightarrow \Z_p^{\times}$ and $\omega_{\cyc} : G_{\Q} \rightarrow \F_p^{\times}$ denote the $p$-adic and mod $p$ cyclotomic character, respectively. We define 
  \[
V_f:=V_f^0(\chi_{\cyc}^{1-\frac{k}{2}}),\ \ A_f:=A_f^0(\chi_{\cyc}^{1-\frac{k}{2}}),\ \ M_f:=M_f^0(\omega_{\cyc}^{1-\frac{k}{2}}),
\vspace{-1mm}\]
and put 
\begin{gather*}
\rho_f:=\rho^0_f\otimes \chi_{\cyc}^{1-\frac{k}{2}} : G_{\Q} \rightarrow \mathrm{Aut}_{\Q_p}(V_f),\\
\bar{\rho}_f := \bar{\rho}^0_f \otimes \omega_{\cyc}^{1-\frac{k}{2}} : G_{\Q} \rightarrow \mathrm{Aut}_{\F_p}(M_f).
\end{gather*}
Hereafter, for a Galois representation $W$ and a character $\chi$ of $G_{\Q}$, $W(\chi)$ denotes the twist of $W$ by $\chi$. Note that we have a $G_{\Q}$-equivariant isomorphism
\begin{eqnarray}\label{selfdual}
V_f \simeq V_f^{\ast} (\chi_{\cyc}),
\end{eqnarray}
by the assumption that $\Q(f)$ is totally real, where $V_f^{\ast}$ denotes the $\Q_p$-linear dual $\Hom_{\Q_p}(V_f, \Q_p)$ of $V_f$.

We consider various quadratic twists of the above Galois representations. Putting $D$ as a quadratic discriminant or 1, we take $\chi_D$ the corresponding quadratic character or the trivial character of $G_{\Q}$, respectively. Then we put \vspace{-1mm}
\[
V_{f, D}:=V_f(\chi_{D}),\ \ A_{f, D}:=A_f(\chi_D),\ \ M_{f, D}:=M_f(\chi_D),
\vspace{-1mm}
\]
and write 
\begin{gather*}
\rho_{f, D}:=\rho_{f} \otimes \chi_{D} : G_{\Q} \rightarrow \mathrm{Aut}_{\Q_p}(V_{f, D}),\\
\bar{\rho}_{f, D}:=\bar{\rho}_f \otimes \chi_{D} : G_{\Q} \rightarrow \mathrm{Aut}_{\F_p}(M_{f, D}).
\end{gather*} 

Finally, as we explained in Section 1, we define a number field $K_{f, D}$ as the fixed field of $\Ker(\bar{\rho}_{f, D})$ in $\overline{\Q}$.

\subsection{The Bloch-Kato Selmer groups and Tate-Shafarevich groups}

Next, we recall the Bloch-Kato Selmer groups and Tate-Shafarevich groups for $p$-adic representations. For more details, see \cite[Section 5]{BK}.

For a number field or a local field $F$, $G_F$ denotes its absolute Galois group $\Gal (\overline{F}/F)$ for a fixed algebraic closure $\overline{F}$. Let $L/F$ be a Galois extension of number fields or local fields, and $N$ a $\Gal(L/F)$-module. We abbreviate the Galois cohomology group $H^i(\Gal(L/F),N)$ to $H^i(L/F, N)$. Moreover, if $L=\overline{F}$, then we abbreviate $H^i(L/F,N)$ to $H^i(F,N)$.  We define the unramified cohomology group $H_{\ur}^i(F, N)$ as a subgroup of cohomology classes in $H^i(F, N)$ which are trivial when restricted to the inertia subgroup at every place of $F$.

Let $V$ be a $p$-adic representation of $G_{\Q}$. We take Bloch and Kato's local condition $H_f^1(\Q_{\ell}, V)$ in $H^1(\Q_{\ell}, V)$ at every prime number $\ell$ as 
\[
   \begin{cases}
    H_f^1(\Q_{\ell}, V):=  H_{\ur}^1(\Q_{\ell}, V)=\Ker\left(H^1(\Q_{\ell}, V)\rightarrow H^1(\Q^{\ur}_{\ell}, V)\right)   
    & (\ell \neq p), \\
    H_f^1(\Q_p, V):= \Ker\left(H^1(\Q_p, V)\rightarrow H^1(\Q_p, V \otimes\mathbf{B}_{\mathrm{crys}})\right) & (\ell=p).
  \end{cases}
\]
Here $\Q^{\ur}_{\ell}$ is the maximal unramified extension of $\Q_{\ell}$, and $\mathbf{B}_{\mathrm{crys}}$ denotes Fontaine's crystalline period ring which is defined in \cite[Section 1]{BK}. Put $T$ as a Galois stable $\Z_p$-lattice in $V$, $A:=V/T$ and $M:=T/pT$. Then we have exact sequences 
\begin{eqnarray}\label{TVA}
0\rightarrow T\xrightarrow{\iota} V \xrightarrow{\pi} A \rightarrow 0,
\end{eqnarray}
\begin{eqnarray}\label{pAA}
0 \rightarrow M \xrightarrow{i} A \xrightarrow{\times p} A \rightarrow 0.
\end{eqnarray}  
By $\pi$ in (\ref{TVA}), we have the induced homomorphism $$\pi : H^1(\Q_{\ell}, V) \rightarrow H^1(\Q_{\ell}, A).$$  Similarly, we also have the induced homomorphism $$i: H^1(\Q_{\ell}, M) \rightarrow H^1(\Q_{\ell}, A)$$ by $i$ in (\ref{pAA}). We define local conditions at $\ell$ with coefficients in $A$ and $M$ as 
\[
H^1_f(\Q_{\ell}, A):=\pi(H^1_f(\Q_{\ell}, V)),\ \ H^1_f(\Q_{\ell}, M):=i^{-1}(H^1_f(\Q_{\ell}, A))
\]
respectively. 

\begin{defn}\label{Sel}
Let $W$ be one of $V$, $A$ and $M$. The Bloch-Kato Selmer group of $W$ is defined as
\[
H^1_f(\Q, W) := \Ker\left(H^1(\Q, W) \xrightarrow{\prod \mathrm{Loc}_{\ell}} \prod_{\ell} \frac{H^1(\Q_{\ell}, W)}{H^1_f(\Q_{\ell}, W)} \right).
\]
Here, $\mathrm{Loc}_{\ell}$ denotes the restriction map $H^1(\Q, W)\rightarrow H^1(\Q_{\ell}, W)$ for every prime number $\ell$ and the product runs over all prime numbers.
\end{defn}

We note a relationship between the Selmer groups $H^1_f(\Q, A)$ and $H^1_f(\Q, M)$. From (\ref{pAA}), we have an exact sequence
\begin{eqnarray}\label{modploc}
0 \rightarrow A^{G_{\Q}}\otimes\F_p\rightarrow H^1(\Q, M) \xrightarrow{i} H^1(\Q, A)[p] \rightarrow 0, 
\end{eqnarray}
where the above map $i$ is the induced one by $i : M \rightarrow A$ in (\ref{pAA}). Then we can see that $$H^1_f(\Q, M)=i^{-1}(H_f^1(\Q, A)[p]).$$

\begin{defn}\label{Shafa}
The $p$-part of the Bloch-Kato Tate-Shafarevich group for $A$ is defined as 
\[
\Sha^{\BK}(\Q, A):=\frac{H^1_f(\Q, A)}{\pi (H^1_f(\Q, V))},
\]
where $$\pi: H^1 (\Q, V) \rightarrow H^1 (\Q, A)$$ is induced by the map $\pi: V \rightarrow A$. 
\end{defn}

By definition, $\Sha^{\BK}(\Q, A)$ sits in the following exact sequence
\begin{eqnarray}\label{seltosha}
0 \rightarrow \pi (H^1_f(\Q, V)) \rightarrow H^1_f(\Q, A) \rightarrow \Sha^{\BK}(\Q, A) \rightarrow 0.
\end{eqnarray}
We note that $\pi (H^1_f(\Q, V))$ is the maximal divisible subgroup of $ H^1_f(\Q, A)$, and hence $\Sha^{\BK}(\Q, A)$ is always finite, in contrast to  a well-known conjecture for classical Tate-Shafarevich groups for elliptic curves. Since $\pi (H^1_f(\Q, V))$ is divisible, we have a surjection $$H^1_f(\Q, A)[p] \rightarrow \Sha^{\BK}(\Q, A)[p]$$ restricting the surjection in (\ref{seltosha}). Thus, we also have a surjection 
\begin{eqnarray}\label{nume}
H^1_f(\Q, M) \rightarrow \Sha^{\BK}(\Q, A)[p],
\end{eqnarray}
 since we have the surjection $i : H^1_f(\Q, M) \rightarrow H^1_f(\Q, A)[p]$ .

\subsection{The main result and a sketch of the proof}
Now we can state our main result in a precise form. 
\begin{thm}[Main result]\label{main}
We assume the following conditions hold. 
\begin{itemize}
\item[{\rm (A)}] $p$ does not divide the level $N$ of $f$.
\item[{\rm (B)}] If $f$ is supersingular at $p$, then $k\leqslant p+1$.\\
 If $f$ is ordinary at $p$, then $p-1\nmid k-1$ and neither of the conditions in Proposition~{\rm \ref{st3}} occur.
\item[{\rm (C)}] The image of $\overline{\rho}^0_{f} : G_{\Q}\rightarrow \mathrm{GL}_2(\F_p)$ contains $\mathrm{SL}_2 (\F_p)$.
\item[{\rm (D)}] For every prime number $\ell \mid N$, the Tamagawa factor $c(\Q_{\ell}, A_{f, D})$ of $A_{f, D}$ is trivial.
\end{itemize}
Then there exists a $\Gal(K_{f, D}/\Q)$-equivariant surjective homomorphism from $\cl(K_{f, D}) \otimes \F_p$ to $M_{f, D}$ if the inequality $\mathrm{dim}_{\F_p}( H_f^1(\Q, M_{f, D})) \geqslant 2$ holds.
\end{thm}
\begin{rem}
In \cite{PS}, Prasad and Shekhar only dealt with the case where $f$ is a rational cusp form of weight 2, and hence corresponds to an elliptic curve over $\Q$ as we mentioned in Section 1. Our main result implies their result at least partially because the assumptions in Theorem \ref{main} are a bit stronger than those in their result. 
\end{rem}
Here, we describe a sketch of the proof of Theorem \ref{main}. We mainly follow the strategy used in \cite{D} in which we gave another generalization of the result of Prasad and Shekhar \cite{PS}.

\vspace{2mm}
\textbf{Step 1}: \textit{We show that a restriction map 
\[
\Res_{K_{f, D}/\Q}: H^1(\Q, M_{f, D}) \rightarrow H^1(K_{f, D}, M_{f, D})^{\Gal(K_{f, D}/\Q)}
\]
is injective under the assumption $(C)$ in Theorem {\rm \ref{main}}.}

\vspace{2mm}
\noindent Assuming the claim in Step 1, we see that 
the above restriction map induces an injection between the unramified cohomology groups 
\[
H_{\ur}^1(\Q, M_{f, D}) \hookrightarrow H_{\ur}^1(K_{f, D}, M_{f, D})^{\Gal(K_{f, D}/\Q)}.
\]
 Due to class field theory, we have $$\Gal(K_{f, D}^{\ur}/K_{f, D})\cong \cl(K_{f, D}),$$ which implies that
\[
H_{\ur}^1(K_{f, D}, M_{f, D})^{\Gal(K_{f, D}/\Q)}=\Hom_{\Gal(K_{f, D}/\Q)}(\cl(K_{f, D}) \otimes \F_p, M_{f, D}).
\]
 Here, the right-hand side is the group of $\Gal(K_{f, D}/\Q)$-equivariant homomorphisms from $\cl(K_{f, D}) \otimes \F_p$ to $M_{f, D}$. Every non-zero homomorphism in it is surjective since the condition $(C)$ implies that $M_{f, D}$ is irreducible as a $\Gal(K_{f, D}/\Q)$-module. From this observation and the above injection between the unramified cohomology groups, it suffices to show $H_{\ur}^1(\Q, M_{f, D})\neq 0$ if $\mathrm{dim}_{\F_p}( H_f^1(\Q, M_{f, D})) \geqslant 2$.

\vspace{2mm}
\textbf{Step 2}: \textit{Under the assumption $(D)$ in Theorem {\rm \ref{main}}, we show that the image of $H_f^1(\Q, M_{f, D})$ in $H^1(\Q_{\ell}^{\ur}, M_{f, D})$ is zero for any $\ell\neq p$. In other words, we show that elements in $H_f^1(\Q, M_{f, D})$ are unramified outside $p$.}

\vspace{2mm}
\noindent Assuming the claim in Step 2, for a restriction map 
\[
\Res^{\ur}_p: H_f^1(\Q, M_{f, D}) \rightarrow H^1(\Q^{\ur}_p, M_{f, D}),
\]
we have $\Ker(\Res^{\ur}_p) \subset H_{\ur}^1(\Q, M_{f, D})$. Thus, it suffices to show that $\Ker(\Res^{\ur}_p) \neq 0$.

\vspace{2mm}
\textbf{Step 3}: \textit{
We show that $\mathrm{dim}_{\F_p} (\mathrm{Im}(\Res^{\ur}_p)) \leqslant 1$ under the assumptions {\rm (A)} and 
{\rm (B)} in Theorem {\rm \ref{main}}.}
\vspace{2mm}

\noindent The above inequality implies $\Ker(\Res^{\ur}_p) \neq 0$ if $\mathrm{dim}_{\F_p}(H^1_f(\Q, M_{f, D}))\geqslant 2$. This completes the proof of Theorem \ref{main}.

\begin{rem}\label{symst}
As we noted in Remark \ref{sym}, we can prove a generalization of Theorem \ref{main} to higher symmetric powers combining the procedure which we explain above and which we used in \cite{D}. More precisely, for $i$ with $0 \leqslant i \leqslant p-2$, we assume the following.
\begin{itemize}
\item[(A${}^{\prime}$)] $p$ does not divide the level $N$ of $f$.
\item[(B${}^\prime$)] If $f$ is supersingular at $p$, then $k\leqslant p+1$.\\
 If $f$ is ordinary at $p$, then $H^0(\Q_p, \mathrm{Sym}^i A_{f, D})\otimes \F_p=0$.
\item[(C${}^{\prime}$)] The representation $\overline{\rho}^0_{f} : G_{\Q}\rightarrow \mathrm{GL}_2(\F_p)$ is surjective.
\item[(D${}^{\prime}$)] For every prime number $\ell \mid N$, the Tamagawa factor $c(\Q_{\ell}, \mathrm{Sym}^i A_{f, D})$ of $\mathrm{Sym}^i A_{f, D}$ is trivial.
\end{itemize}
Then we can show that there exists a $\Gal(K_{f, D}/\Q)$-equivariant surjection $$\cl(K_{f, D}) \otimes \F_p \twoheadrightarrow\mathrm{Sym}^i M_{f, D}$$ if the inequality $\mathrm{dim}_{\F_p}( H_f^1(\Q, \mathrm{Sym}^i M_{f, D})) \geqslant i+1$ holds. When $i=1$ and $k>2$, the above condition (B${}^\prime$) is equivalent to (B) in 
Theorem~\ref{main}. 
\end{rem}

\section{Injectivity of the restriction map $\Res_{K_{f, D}/\Q}$}
In this section, we prove the claim of Step 1. In the following, we omit the suffixes of our $V_{f, D}, T_{f, D}, A_{f, D}, M_{f, D}$, $K_{f, D}$ as $V, T, A, M, K$ if no confusion occurs.

\begin{prop}\label{res}
Suppose that $\mathrm{Im}(\bar{\rho}^0_f)$ contains $\mathrm{SL}_2(\F_p)$ {\rm (}the assumption {\rm (C)} in 
Theorem~{\rm \ref{main})}. Then $\Res_{K/\Q}: H^1(\Q, M) \rightarrow H^1(K, M)^{\Gal(K/\Q)}$
is injective. 
\end{prop}

\begin{proof}[Proof of Proposition {\rm \ref{res}}]
Due to the inflation-restriction exact sequence, the kernel of the restriction map $\Res_{K/\Q}$ is $H^1(K/\Q, M)$. Thus, it suffices to show that $H^1(K/\Q, M)=0$. We use the following lemma.
\begin{lem}\label{vanish}
Let $G$ be a finite group and $N$ a finite dimensional representation of $G$ over $\F_p$. If there exists a normal subgroup $H$ of $G$ such that $(1)$ $\#H$ is prime to  $p$ and $(2)$ $N^{H}=0$, then  $H^i(G, N)=0$ for all $i\geqslant 0$.
\end{lem}

For a proof of this lemma, see \cite[Lemma 3.2]{D}. \\

We divide the proof of Proposition~\ref{res} into two cases.

(Case 1 : $p\geqslant 5$)\ \ Let $L:=\Q(\zeta_p, \sqrt{D})$ where $\zeta_p$ is a primitive $p$-th root of unity. First, we show that the image of $\bar{\rho}^0_f: G_{\Q} \rightarrow \mathrm{GL}_2(\mathbb{F}_p)$ still contains $\mathrm{SL}_2(\F_p)$ when restricted to $G_L$. Let $F$ and $F_L$ be the fields corresponding to the kernel of $\bar{\rho}^0_f$ and $\bar{\rho}^0_f|_{G_L}$ in Galois theory respectively. We have $\Gal(F/\Q) \cong \mathrm{Im}(\bar{\rho}^0_f) \supset \mathrm{SL}_2(\F_p)$ and $\Gal(F/F\cap L) \cong \Gal(F_L/L)\cong\mathrm{Im}(\bar{\rho}^0_f|_{G_L})$. Let $F^{\prime}$ be the intermediate field of the extension $F/\mathbb{Q}$ corresponding to $\mathrm{SL}_2(\F_p)$. Then $F^{\prime}\cdot (F\cap L)/F^{\prime}$ is an abelian extension as the extension $F\cap L/\Q$ is. On the other hand, $\mathrm{SL}_2(\F_p)$ is a perfect group since we assume $p\geqslant 5$. In other words, $\mathrm{SL}_2(\F_p)$ has no non-trivial abelian quotients. Hence, we have $F^{\prime}\cdot (F\cap L)=F^{\prime}$ and $\mathrm{SL}_2(\F_p)\cap \Gal(F/F\cap L)=\mathrm{SL}_2(\F_p)$. Thus, $\mathrm{Im}(\bar{\rho}^0_{f}|_{G_L}) \cong \Gal(F/F\cap L) \supset \mathrm{SL}_2(\F_p)$. Since $\bar{\rho}_{f, D}|_{G_L}=\overline{\rho}^0_f|_{G_L}$, $\mathrm{Im}(\bar{\rho}_{f, D})\cong\Gal(K/\Q)$ also contains $\mathrm{SL}_2(\F_p)$ and especially  $-I$, where $I$ denotes the unit matrix in $\mathrm{SL}_2(\F_p)$. Then the subgroup $H$ of $\Gal(K/\Q)$ generated by $-I$ satisfies the conditions (1) and (2) in Lemma \ref{vanish} for $N=M$ and we have $H^i(K/\Q, M)=0$ for all $i\geqslant 0$.

(Case 2 : $p=3$)\ \ Since we assume that $\mathrm{Im}(\bar{\rho}^0_f) \supset \mathrm{SL}_2(\F_{3})$, $\mathrm{Im}(\bar{\rho}^0_f)$ contains the matrices
\[
A_1=\left(
 \begin{array}{cc}
 1&1\\
 1&-1
 \end{array}
 \right) , 
 A_2=\left(
 \begin{array}{cc}
 -1&-1\\
 -1&1
 \end{array}
 \right)
\] 
for a suitable basis of $M$ over $\F_3$. We take $\sigma \in G_{\Q}$ such that $\bar{\rho}^0_f(\sigma)=A_1$. Since both $\omega_{\cyc}$ and $\chi_{D}$ have order 2 in this case, $\bar{\rho}_{f, D}(\sigma)=A_1$ or $A_2$. We can show that there are no non-trivial elements in $M$ which are fixed by $A_1$ or $A_2$ by direct computation. Hence, if $\bar{\rho}_{f, D}(\sigma)=A_1 \in \Gal(K/\Q)$ (resp. $A_2$), then the subgroup of $\Gal(K/\Q)$ generated by $A_1$ (resp. $A_2$) satisfies the conditions $(1), (2)$ in Lemma \ref{vanish} for $N=M$ since every subgroup of ${\rm{SL}}_2(\F_3)$ is normal and $2$-group. Thus, we have $H^i(K/\Q, M)=0$ for all $i\geqslant 0$ from Lemma \ref{vanish}. 
\end{proof}

\section{Localization of the Selmer group $H^1_f(\mathbb{Q}, M)$ outside $p$}
\subsection{Tamagawa factor}
Here, we introduce some basic notions on the Tamagawa factor. In this subsection, $V$, $A$ and $M$ denotes a general $p$-adic Galois representations as in Subsection 2.2.

Fixing a prime number $\ell \neq p$, we write $I_{\ell}$ as the inertia subgroup of $G_{\Q_{\ell}}$. We define a finite group $\mathcal{A}:=A^{I_{\ell}}/ (A^{I_{\ell}})_{\mathrm{div}}$, where $(A^{I_{\ell}})_{\mathrm{div}}$ denotes the maximal divisible subgroup of $A^{I_{\ell}}$. 
\begin{defn}
We define the $p$-part of the Tamagawa factor of $A$ at $\ell$ as 
\[
c(\Q_{\ell}, A) := \#\mathcal{A} /(\Frob_{\ell}-1)\mathcal{A}=\#\mathcal{A}^{\Frob_{\ell}=1}.
\]
Here $\Frob_{\ell}$ denotes the Frobenius element in $\Gal(\Q^{\ur}_{\ell}/\Q_{\ell})$. 
\end{defn}

Note that for an endomorphism $g$ of $\mathcal{A}$, the orders of $\Ker(g)$ and $\mathrm{Coker}(g)$ are the same since $\mathcal{A}$ is finite. Hence, the second equality in the definition holds. Roughly speaking, the Tamagawa factor $c(\Q_{\ell}, A)$ can be used to measure the difference between $H^1_{\ur} (\Q_{\ell}, M)$ and the local condition $H^1_f(\Q_{\ell}, M)$.
\begin{prop}\label{urfinite}
If $c(\Q_{\ell}, A)=1$, then $H^1_f(\Q_{\ell}, M)=H^1_{\ur} (\Q_{\ell}, M)$.
\end{prop}
\begin{proof}
We first consider the local condition with coefficients in $A$.
\begin{eqnarray*}
H^1_{\ur}(\Q_{\ell}, A)/H^1_f(\Q_{\ell}, A) & \overset{\sim}{\longrightarrow} & \mathrm{Coker}\left(H^1_f(\Q_{\ell}, V)=H^1_{\ur}(\Q_{\ell}, V) \overset{\pi}{\longrightarrow} H^1_{\ur}(\Q_{\ell}, A)\right)\\
&\overset{\sim}{\longrightarrow}& \mathrm{Coker}\left(V^{I_{\ell}}/(\Frob_{\ell}-1)(V^{I_{\ell}}) \overset{\pi}{\longrightarrow} A^{I_{\ell}}/(\Frob_{\ell}-1)(A^{I_{\ell}})\right)\\
&\overset{\sim}{\longrightarrow}& \mathcal{A}/(\Frob_{\ell}-1)\mathcal{A}.
\end{eqnarray*}
Here, we use the identification $$H^1_{\ur}(\Q_{\ell}, V) = H^1(\Q^{\ur}_{\ell}/\Q_{\ell}, V^{I_{\ell}})\simeq V^{I_{\ell}}/(\Frob_{\ell}-1)(V^{I_{\ell}})$$ in the second isomorphism above. Thus, if $c(\Q_{\ell}, A)=1$, we have $$H^1_{\ur}(\Q_{\ell}, A)=H^1_f(\Q_{\ell}, A).$$
On the other hand, since $H^1_f(\Q_{\ell}, M)$ is the inverse image of $H^1_f(\Q_{\ell}, A)=H^1_{\ur}(\Q_{\ell}, A)$ under $i: H^1_f(\Q_{\ell}, M) \rightarrow H^1_f(\Q_{\ell}, A)$, we have  
\begin{eqnarray*}
H^1_f(\Q_{\ell}, M)&=&\mathrm{ker}(H^1(\Q_{\ell}, M) \rightarrow H^1(\Q_{\ell}, A) \rightarrow H^1(\Q^{\ur}_{\ell}, A)^{\Frob_{\ell}=1})\\
&=& \ker(H^1(\Q_{\ell}, M) \rightarrow H^1(\Q^{\ur}_{\ell}, M)^{\Frob_{\ell}=1} \overset{g}{\rightarrow} H^1(\Q^{\ur}_{\ell}, A)^{\Frob_{\ell}=1}).
\end{eqnarray*}
We can see that the above homomorphism $g$ is injective. In fact, from the exact sequence (\ref{pAA}), we have an exact sequence
\[
0\rightarrow (A^{I_{\ell}} \otimes \F_p)^{\Frob_{\ell}=1} \longrightarrow H^1(\Q^{\ur}_{\ell}, M)^{\Frob_{\ell}=1}\overset{g}{\longrightarrow} H^1(\Q^{\ur}_{\ell}, A)^{\Frob_{\ell}=1}.
\]
Since $A^{I_{\ell}} \otimes \F_p = \mathcal{A}\otimes \F_p$, we obtain
\[
\Ker(g)=(A^{I_{\ell}} \otimes \F_p)^{\Frob_{\ell}=1}=(\mathcal{A} \otimes \F_p)^{\Frob_{\ell}=1}=\mathcal{A}^{\Frob_{\ell}=1} \otimes \F_p=0.
\]
In the third equality above, we use the assumption $c(\Q_{\ell}, A)= \#\mathcal{A}^{\Frob_{\ell}=1}=1$ to get $H^1(\Q^{\ur}_{\ell}/\Q_{\ell}, \mathcal {A})=0$. Thus the homomorphism $g$ is injective and
\[
H^1_f(\Q_{\ell}, M)= \Ker(H^1(\Q_{\ell}, M) \rightarrow H^1(\Q^{\ur}_{\ell}, M)^{\Frob_{\ell}=1} )=H^1_{\ur}(\Q_{\ell}, M). 
\] 
The proof of Proposition \ref{urfinite} is complete.
\end{proof}

We introduce a sufficient condition on $M=A[p]$ for $c(\mathbb{Q}_{\ell}, A)=1$.
\begin{prop}\label{suftam}
For a prime number $\ell \neq p$, if $M^{G_{\Q_{\ell}}}=0$, then $c(\Q_{\ell}, A)=1$.
\end{prop}
\begin{proof}
We have a commutative diagram of exact sequences
\[
  \xymatrix{
 0 \ar[r]^{} &  (A^{I_{\ell}})_{\mathrm{div}}\ar[r]\ar[d]^{\Frob_{\ell} -1} &   A^{I_{\ell}} \ar[r]^{}\ar[d]^{\Frob_{\ell} -1} & \mathcal{A}\ar[r]^{} \ar[d]^{\Frob_{\ell} -1}& 0\\
 0 \ar[r]^{} &  (A^{I_{\ell}})_{\mathrm{div}}\ar[r] &   A^{I_{\ell}} \ar[r]^{} & \mathcal{A}\ar[r]^{} & 0.
  }
\]
The assumption $M^{G_{\Q_{\ell}}}=0$ implies $A^{G_{\Q_{\ell}}}=0$. In fact, if we have a non-trivial element $x \in A^{G_{\Q_{\ell}}}$, then the element $p^mx$ is a non-trivial element in $M^{G_{\Q_{\ell}}}=A[p]^{G_{\Q_{\ell}}}$ for sufficiently large $m \in \Z$. Hence, the middle vertical arrow is injective and so is the left vertical arrow. While $(A^{I_{\ell}})_{\mathrm{div}}$ is the direct sum of finite numbers of $\Q_p/\Z_p$, the injective left vertical arrow is in fact an isomorphism. Therefore, the right vertical arrow is injective by Snake lemma which implies $c(\Q_{\ell}, A)=\#\mathcal{A}^{\Frob_{\ell} =1}=1$.
\end{proof}

\subsection{Unramifiedness of $H^1_f(\Q, M_{f, D})$ outside $p$}
In this subsection, we reset $A=A_{f, D}$ and $M=M_{f, D}$. Now we prove the following proposition which is the claim in Step 2. 
\begin{prop}\label{st2}
If $c(\Q_{\ell}, A)=1$ for every prime $\ell$ with $\ell\mid N$ {\rm (}the assumption {\rm (D)} in Theorem~{\rm \ref{main})}, then the elements in $H_f^1(\Q, M)$ are unramified outside $p$.
\end{prop}
\begin{proof}
Since $p\geqslant 3$, there is nothing to prove for unramifiedness at the infinite place of $\Q$. If a prime number $\ell\ (\neq p)$ does not divide $N$, the inertia subgroup $I_{\ell}$ acts on $A^0_f$ trivially. So we have 
 \[
A_{f, D}^{I_{\ell}}=
\begin{cases}
A_{f, D} & (\sqrt{D} \in  \Q^{\ur}_{\ell}),\\
0 & (\mathrm{otherwise}),
\end{cases}
\] 
which implies $c(\Q_{\ell}, A_{f, D})=1$  in both cases. While for $\ell\mid N$, we assume that $c(\Q_{\ell}, A)=1$. Hence, we have $H^1_f(\Q_{\ell}, M)=H^1_{\mathrm{ur}}(\Q_{\ell}, M)$ for all $\ell \neq p$ from Proposition \ref{urfinite}. Thus, elements in the Selmer group $H^1_f(\Q, M)$ are unramified outside $p$.
\end{proof}

\section{Localization of the Selmer group $H^1_f(\Q, M_{f, D})$ at $p$ }
In this section, we consider the restriction map $\Res_p^{\ur}:H_f^1(\Q, M) \rightarrow H^1(\Q^{\ur}_p, M)$ at $p$. We show that its image has at most dimension 1 over $\F_p$ under the assumptions (A) and (B) in Theorem~\ref{main}. We decompose $\Res_p^{\ur}$ as
\[
\Res_p^{\ur}:H_f^1(\Q, M) \xrightarrow{\mathrm{Loc}_p} H^1(\Q_p,M) \xrightarrow{\Res_{\Q_p^{\ur}/\Q_p}}  H^1(\Q^{\ur}_p, M).
\]
Here, $\Res_{\Q_p^{\ur}/\Q_p}$ denotes the restriction of cohomology classes to the inertia subgroup at $p$. We first study the image of $H^1_f(\Q, M)$ under $\mathrm{Loc}_p$ which is a subspace of the local condition $H^1_f(\Q_p, M)$. By the definition of the local condition $H^1_f(\Q_p, M)$ and (\ref{pAA}), we have an exact sequence 
\begin{eqnarray*}\label{locexact}
0 \rightarrow  \displaystyle  A^{G_{\Q_p}}\otimes \F_p \rightarrow H_f^1(\Q_p, M)\xrightarrow{\iota}  H_f^1(\Q_p, A)[p] \rightarrow 0.
\end{eqnarray*}
 Therefore, we have an inequality
 \begin{eqnarray}\label{fundineq}
 \mathrm{dim}_{\F_p}(\mathrm{Im}(\mathrm{Loc}_p))&\leqslant& \mathrm{dim}_{\F_p}(H^1_f(\Q_p, M))\nonumber\\
  &\leqslant& \displaystyle \mathrm{dim}_{\F_p}\left(A^{G_{\Q_p}}\otimes \F_p\right)+\mathrm{dim}_{\F_p}\left( H_f^1(\Q_{p}, A)[p] \right).
 \end{eqnarray}
 Since $H_f^1(\Q_{p}, A)$ is the image of $H_f^1(\Q_{p}, V)$ and cofinitely generated as a $\Z_p$-module, we have 
 $$\mathrm{dim}_{\F_p}( H_f^1(\Q_p, A)[p] )=\mathrm{dim}_{\Q_p}(H_f^1(\Q_{p}, V)).$$
 The dimension of $H_f^1(\Q_{p}, V)$ can be computed by the following fact.
 \begin{prop}[\cite{BK}, Corollary 3.8.4]
 For a $p$-adic representation $V$ of $G_{\Q_p}$, we define 
 \begin{gather*}
 \mathbf{D}_{\mathrm{dR}}(V):=(V\otimes \mathbf{B}_{\mathrm{dR}})^{G_{\Q_p}}, \\
\mathbf{D}^{+}_{\mathrm{dR}}(V):=(V\otimes \mathbf{B}^{+}_{\mathrm{dR}})^{G_{\Q_p}}. 
\end{gather*}
Here, $\mathbf{B}_{\mathrm{dR}}$ is Fontaine's de Rham period ring and $\mathbf{B}^{+}_{\mathrm{dR}}$ is its valuation ring. If $V$ is a de Rham representation, then
 \begin{eqnarray}\label{hodge}
 \mathrm{dim}_{\Q_p}(H_f^1(\Q_p, V))=  \mathrm{dim}_{\Q_p}(\mathbf{D}_{\mathrm{dR}}(V)/ \mathbf{D}^{+}_{\mathrm{dR}}(V))+{\mathrm{dim}}_{\Q_p}H^0(\Q_p, V).
 \end{eqnarray}
 \end{prop}
 \begin{lem}\label{tang}
 \[
 \mathrm{dim}_{\Q_p}(\mathbf{D}_{\mathrm{dR}}(V_{f, D})/ \mathbf{D}^{+}_{\mathrm{dR}}(V_{f, D}))=1.
 \]
 \end{lem}
 \begin{proof}
 Since we assume $p\nmid N$, $V^0_f$ is a crystalline, especially a de Rham representation when restricted to $G_{\Q_p}$. We know that $\Q_p(\chi_{\cyc}^{1-k/2})$ and quadratic character $\chi_{D}$ are also de Rham as representations of $G_{\Q_p}$. Since tensor products of de Rham representations are also de Rham, $V_{f, D}=V_f^0 \otimes \Q_p(\chi_{\cyc}^{1-k/2})\otimes \chi_D$ is a de Rham representation, and therefore we obtain 
 $$\mathrm{dim}_{\Q_p}\mathbf{D}_{\mathrm{dR}}(V_{f, D})=\mathrm{dim}_{\Q_p}(V_{f, D})=2.$$

On the other hand, the Hodge-Tate weights of $V_f^0$ are \{0, $k-1$\} as we mentioned in Subsection \ref{notation}. Note that twisting a $p$-adic representation $\rho$ of $G_{\Q_p}$ by $\chi_D$ does not change its Hodge-Tate weights, while twisting $\rho$ by $\chi_{\cyc}$ shifts them by $+1$. Hence, the Hodge-Tate weights of $V_{f, D}$ are $\{1-k/2, k/2\}$, which implies that ${\mathrm{dim}}_{\Q_p}(\mathbf{D}^{+}_\mathrm{dR}(V))=1$. Thus, we obtain 
$${\mathrm{dim}}_{\Q_p}(\mathbf{D}_{\mathrm{dR}}(V)/ \mathbf{D}^{+}_{\rm{dR}}(V))=2-1=1.$$ 
This ends the proof of Lemma \ref{tang}.
\end{proof}

 Using (\ref{hodge}) and Lemma \ref{tang}, we see that
 \begin{eqnarray}\label{ineq2}
 \mathrm{dim}_{\F_p}( H_f^1(\Q_p, A)[p]) =\mathrm{dim}_{\Q_p}(H_f^1(\Q_{p}, V))=1 + \mathrm{dim}_{\Q_p}H^0(\Q_p, V).
 \end{eqnarray}
 By (\ref{fundineq}) and (\ref{ineq2}), we deduce that 
 \begin{eqnarray}\label{ineq}
  \mathrm{dim}_{\F_p}(\mathrm{Im}(\mathrm{Loc}_p)) \leqslant \displaystyle \mathrm{dim}_{\F_p}\left(A^{G_{\Q_p}}\otimes \F_p\right)+1+ \mathrm{dim}_{\Q_p}H^0(\Q_p, V).
 \end{eqnarray}
In the following, we compute the first and the third terms in the above inequality. 
\subsection{Supersingular case}
\begin{prop}\label{compss}
If $f$ is supersingular at $p$ and $k\leqslant p+1$, then we have 
\[
A^{G_{\Q_p}}\otimes \F_p=V^{G_{\Q_p}}=0.
\]
\end{prop}
We use the following result of Fontaine and Edixhoven.
\begin{thm*}[Fontaine-Edixhoven, \cite{Edi}, Theorem 2.6]\mbox{}\\
Suppose that $f$ is supersingular at $p$, and that the weight $k$ of $f$ satisfies $2\leqslant k\leqslant p+1$. Then $\bar{\rho}_f^0|_{G_{\mathbb{Q}_p}}$ is irreducible.
\end{thm*}
\begin{rem}
Actually, their result describes a more precise local behavior of $\bar{\rho}_f^0$ at a supersingular prime. See for example \cite{Edi}.
\end{rem}

\begin{proof}[Proof of Proposition {\rm \ref{compss}}]
From the theorem of Fontaine and Edixhoven, $M^0_f$ and its twist $M=M_{f, D}$ are irreducible $G_{\Q_p}$-modules. Hence, we obtain $M^{G_{\Q_p}}=0$ which implies  $A^{G_{\Q_p}}=0$ by the same argument as in the proof of Lemma \ref{suftam}. Thus, we obtain $A^{G_{\Q_p}}\otimes \F_p=0$. Since the residual representation $M$ is irreducible, the $p$-adic representation $V$ is also irreducible as a $G_{\Q_p}$-module which implies $V^{G_{\Q_p}}=0$.
\end{proof}

\subsection{Ordinary case}
\begin{prop}\label{compV}
Suppose $f$ is ordinary at $p$. Then we have $V^{G_{\Q_p}}=0.$
\end{prop}
\begin{proof}
We know $g \in G_{\Q_p}$ acts on $V=V_{f, D}=V_f^0(\chi_{\cyc}^{1-k/2}\otimes\chi_D)$ as 
\begin{eqnarray}\label{actV}
 \left(
 \begin{array}{cc}
 \chi_{\cyc}^{k/2}(g)\psi^{-1}(g) &u(g) \\
 0&\psi(g)\chi_{\cyc}^{1-k/2}
 \end{array}
 \right)\cdot \chi_{D}(g)
\end{eqnarray}
for a suitable basis of $V$. Here, $\psi: G_{\Q_p} \rightarrow \mathbb{Z}^{\times}_p$ is an unramified character and $u(g) \in \Z_p$. We have a subspace $W_1$ of $V$ on which $G_{\Q_p}$ acts via the character $\chi_{\cyc}^{k/2}\chi_{D}\psi^{-1}$. If $V^{G_{\Q_p}}\neq 0$, then $V^{G_{\Q_p}}$ is 1-dimensional and linearly independent with $W_1$ since $\chi_{\cyc}^{k/2}\chi_{D}\psi^{-1}\neq 1$. Then the matrix (\ref{actV}) is similar to 
\begin{eqnarray*}
 \left(
 \begin{array}{cc}
 \chi_{\cyc}^{k/2}(g)\chi_D\psi^{-1}(g) &0 \\
 0&1
 \end{array}
 \right).
\end{eqnarray*}
Taking determinants of above two matrices, we have $\chi_{\cyc} =  \chi_{\cyc}^{k/2}\chi_D\psi^{-1}$ which implies $\chi_{\cyc}^{1-k/2}\psi=\chi_D$. However, this is a contradiction since $\chi_{\cyc}$ is a ramified infinite order character and $\psi$ is an unramified infinite order character. Hence, we obtain $V^{G_{\Q_p}}=0$.
\end{proof}

 \begin{prop}\label{comp}
 Suppose that $f$ is ordinary at $p$ and $p-1\nmid k-1$. Then we have $M^{G_{\Q_p}}=0$ if and only if neither of the following situations  happen.

 \noindent$($When $p\nmid D)$
 \begin{itemize}
  \item[{\rm (i)}] $M$ splits as a $G_{\Q_p}$-module, $p-1$ divides one of $k/2$ or $1-k/2$, 
 
 \hspace{-6mm}and $a_p(f) \equiv 1 \pmod p$.
 \item[{\rm (ii)}] $M$ does not split as a $G_{\Q_p}$-module, $p-1\mid k/2$ and $a_p(f) \equiv 1 \pmod p$.
 \end{itemize}
 \noindent$($When $D=p^{\ast}:=(-1)^{\frac{p-1}{2}}p)$
 \begin{itemize}
 \item[{\rm (iii)}] $M$ splits as a $G_{\Q_p}$-module, $p-1$ divides one of $\frac{k-p+1}{2}$ or $\frac{k+p-3}{2}$,
 
 \hspace{-7mm} and $a_p(f) \equiv 1 \pmod p$.
 \item[{\rm (iv)}] $M$ does not split as a $G_{\Q_p}$-module, $p-1 \mid \frac{k-p+1}{2}$ and $a_p(f) \equiv 1 \pmod p$.
  \end{itemize}
  Here, $a_p(f)$ is the $p$-th Fourier coefficient of $f$.
  \end{prop}

\begin{proof}
 Taking a suitable basis \{$m_1, m_2$\} for $M=M_{f, D}$ over $\F_p$, we know $g \in G_{\Q_p}$ acts on $M$ as 
\begin{eqnarray}\label{actM}
 \left(
 \begin{array}{cc}
 \omega_{\cyc}^{k/2}(g)\bar{\psi}^{-1}(g) &\bar{u}(g) \\
 0&\bar{\psi}(g)\omega_{\cyc}^{1-k/2}(g)
 \end{array}
 \right)\cdot \chi_{D}(g),
\end{eqnarray}
where $\psi$ and $u(g) \in \Z_p$ are as in (\ref{actV}) and $\bar{\psi}$ and $\bar{u}(g)$ are their reduction mod $p$. We show Proposition \ref{comp} by a case-by-case computation.

\vspace{1mm}
(Case 1 : $p\nmid D$)\ \ In this case, we have $\sqrt{D}\in \Q^{\ur}_p$. By (\ref{actM}), an element $g$ in the inertia subgroup $I_p$ at $p$ acts on $M$ via a matrix
\begin{eqnarray*}
 \left(
 \begin{array}{cc}
 \omega_{\mathrm{cyc}}^{k/2}(g) &\bar{u}(g) \\
 0&\omega_{\mathrm{cyc}}^{1-k/2}(g)
 \end{array}
 \right).
\end{eqnarray*}

First, suppose that $M$ splits as a $G_{\mathbb{Q}_p}$-module, or in other words, suppose $\bar{u}=0$ holds. For $a, b \in \F_p$, an element $x=am_1+bm_2\in M$ and $g \in I_p$, we have
\[
g(x)=g(am_1+bm_2)=a\omega_{\cyc}^{k/2}(g)m_1+b\omega_{\cyc}^{1-k/2}(g)m_2.
\]
Thus, $x \in M^{I_p}$ if and only if 
\begin{eqnarray}\label{invcondisplitcase}
a\omega_{\mathrm{cyc}}^{k/2}(g)=a,\ b\omega_{\mathrm{cyc}}^{1-k/2}(g)=b\ \ \text{for all $g \in I_p$}.
\end{eqnarray}
If $p-1\nmid k/2$ and $p-1 \nmid 1-k/2$, (\ref{invcondisplitcase}) implies $a=b=0$ and $M^{G_{\Q_p}}\subset M^{I_p}=0$. Suppose $p-1\mid k/2$ or $p-1 \mid 1-k/2$ holds. Then we obtain $M^{I_p}=\F_p m_1$ or $\F_p m_2$ respectively. The Frobenius element $\Frob_p$ acts on $\F_p m_1$ and $\F_p m_2$ via multiplication by $a_p(f)^{-1}$ mod $p$ and $a_p(f)$ mod $p$ respectively, where $a_p(f)$ is the $p$-th Fourier coefficient of $f$. Thus, $M^{G_{\Q_p}}= (M^{I_p})^{\Frob_p=1}=0$ if $a_p(f) \not\equiv 1 \pmod p$. When $a_p(f) \equiv 1 \pmod p$, which is the situation (i) in Proposition \ref{comp}, then $M^{G_{\Q_p}}=\F_p m_1$ or $\F_p m_2$. \\

Next, suppose $M$ does not split as a $G_{\Q_p}$-module, we use the following lemma.
\begin{lem}\label{nonsplitproperty}
Suppose that $M$ does not split as a $G_{\Q_p}$-module. If $p-1\nmid k-1$, then $M$ still does not split when restricted to $G_{\Q^{\mathrm{ab}}_p}$
, where $\Q^{\mathrm{ab}}_p$ denotes the maximal abelian extension of $\Q_p$.
\end{lem}

We prove this lemma later. From this lemma and its proof, we can retake a suitable basis of $M$ such that $G_{\Q_p}$ acts on $M$ also as (\ref{actM}) and $\bar{u}(G_{\Q^{\mathrm{ab}}_p})\neq 0$. Then $g \in G_{\Q^{\mathrm{ab}}_p}$ acts on $M$ via $
 \left(
 \begin{array}{cc}
 1 &\bar{u}(g) \\
 0&1
 \end{array}
 \right).
$
Hence, we can show that an element $x=am_1+bm_2 \in M$ is fixed by $G_{\Q^{\mathrm{ab}}_p}$ if and only if 
\[
a=a+b\bar{u}(g) \ \ \text{for all $g \in G_{\Q^{\mathrm{ab}}_p}$}.
\]
Since $\bar{u}(G_{\Q^{\mathrm{ab}}_p})\neq 0$, we obtain $b=0$ and $M^{I_p}=(M^{G_{\Q^{\mathrm{ab}}_p}})^{I_p}=(\F_p m_1)^{I_p}$. Here, $I_p$ acts on $\F_p m_1$ via the character $\omega_{\cyc}^{k/2}$ by (\ref{actM}). Hence, if $p-1 \nmid k/2$, then  $M^{G_{\Q_p}}\subset M^{I_p}=0$. Suppose $p-1 \mid k/2$. Then $M^{I_p}=\F_p m_1$ and $\Frob_p$ acts on this space via multiplication by $a_p(f)^{-1}$ mod $p$. Thus, if $a_p(f) \not\equiv 1 \pmod p$, we have $M^{G_{\Q_p}}= (M^{I_p})^{\Frob_p=1}=0$. When $a_p(f) \equiv 1 \pmod p$, which is the situation (ii), then $M^{G_{\Q_p}}=\F_p m_1$.

(Case 2 : $p\mid D$ and $D\neq p^{\ast}$)\ \ In this case, $\Q^{\ur}_p(\sqrt{D})$ and $\Q^{\mathrm{ur}}_p(\zeta_p)$ are linearly disjoint over $\Q^{\ur}_p$. Hence, there exists $\sigma \in G_{\Q^{\ur}_p(\zeta_p)}$ such that $\chi_{D}(\sigma)=-1$. By (\ref{actM}), $g \in G_{\Q^{\ur}_p(\zeta_p)}$ acts on $M$ via 
\begin{eqnarray*}
 \left(
 \begin{array}{cc}
 1 &\bar{u}(g) \\
 0&1
 \end{array}
 \right)\cdot \chi_{D}(g).
\end{eqnarray*}
Thus, an element $x=am_1+bm_2\in M$ fixed by $G_{\Q^{\ur}_p(\zeta_p)}$ if and only if 
\[
\chi_D(g)(a+b\bar{u}(g))=a ,\ \chi_D(g) b=b \ \ \text{for all $g \in G_{\Q^{\ur}_p(\zeta_p)}$}.
\]
Putting $g=\sigma$, we obtain $b=0$ from the second equality since $p$ is odd and then we get $a=0$ from the first equality. Thus, we have $M^{G_{\Q_p}}\subset M^{G_{\Q^{\ur}_p(\zeta_p)}}=0$.

(Case 3 : $D=p^{\ast}$)\ \ In this case, we have an inclusion $\Q^{\ur}_p(\sqrt{D})\subset \Q^{\ur}_p(\zeta_p)$. Suppose that $M$ splits as a $G_{\Q_p}$-module. As the calculation in (Case 1), an element $x=am_1+bm_2 \in M$ is fixed by $I_p$ if and only if 
\begin{eqnarray}\label{invcondisplit}
a\omega_{\cyc}^{k/2}(g)\chi_D(g)=a,\ b\omega_{\cyc}^{1-k/2}(g)\chi_D(g)=b\ \ \text{for all $g \in I_p$}.
\end{eqnarray}
This condition is equivalent to the condition that the same equations hold for a generator $\tau$ of the Galois group $\Gal(\Q^{\ur}_p(\zeta_p)/\Q^{\ur}_p)$ since both $\omega_{\cyc}$ and $\chi_D$ are trivial on $G_{\Q_p^{\ur}(\zeta_p)}$. We know that $\chi_D(\tau)=-1$ and $\omega_{\cyc}(\tau)$ has order $p-1$. Thus, (\ref{invcondisplit}) is equivalent to the condition
\[
-a\omega_{\cyc}^{k/2}(\tau)=a,\ -b\omega_{\cyc}^{1-k/2}(\tau)=b.
\]
If $\frac{p-1}{2}\not\equiv k/2 \pmod{p-1}$ and $\frac{p-1}{2}\not\equiv 1-k/2 \pmod{p-1}$, then $a=b=0$ from the above condition and we get $M^{G_{\Q_p}}\subset M^{I_p}=0$. Suppose $\frac{p-1}{2}\equiv k/2 \pmod{p-1}$ or $\frac{p-1}{2}\equiv 1-k/2 \pmod{p-1}$ holds. Then $M^{I_p}=\F_p m_1$ or $\F_p m_2$ respectively. As in the argument in (Case 1), if $a_p(f) \not\equiv 1 \pmod p$, then $M^{G_{\Q_p}}= (M^{I_p})^{\Frob_p=1}=0$. When $a_p(f) \equiv 1 \pmod p$, which is the situation (iii), $M^{G_{\Q_p}}=\F_p m_1$ or $\F_p m_2$. 

Finally, suppose that $M$ does not split as a $G_{\Q_p}$-module. By the same argument as in (Case 1), we can show that an element $x=am_1+bm_2 \in M$ is fixed by $G_{\Q^{\mathrm{ab}}_p}$ if and only if 
\[
a=a+b\bar{u}(g) \ \ \text{for all $g \in G_{\Q^{\mathrm{ab}}_p}$}.
\]
This implies that $b=0$ from Lemma \ref{nonsplitproperty} after retaking a suitable basis of $M$. Then 
$$M^{I_p}=(M^{G_{\Q^{\mathrm{ab}}_p}})^{I_p}=(\F_p m_1)^{I_p}$$ and $I_p$ acts on $\F_p m_1$ via $\omega_{\cyc}^{k/2}\chi_D$ by (\ref{actM}). So we can see that $x=am_1\in M^{G_{\Q^{\mathrm{ab}}_p}}$ is fixed by $I_p$ if and only if 
\[
-a \omega_{\cyc}^{k/2}(\tau)=a.
\]
 If $\frac{p-1}{2}\not\equiv k/2 \pmod{p-1}$, then $a=0$ from this equality and we obtain 
 $$M^{G_{\Q_p}}\subset M^{I_p}=0.$$
 Suppose that $\frac{p-1}{2}\equiv k/2 \pmod{p-1}$. Then $M^{I_p}=\F_p m_1$ and we have 
 $$M^{G_{\Q_p}}= (M^{I_p})^{\Frob_p=1}=(\F_p m_1)^{\Frob_p=1}.$$
 This is trivial if $a_p(f) \not\equiv 1 \pmod p$, or otherwise we have $M^{G_{\Q_p}}=\F_p m_1$ (situation (iv)).
 The proof of Proposition \ref{comp} is complete. 
 \end{proof}

Finally, we prove Lemma \ref{nonsplitproperty} by using the following fact.

\begin{prop}[\cite{LR}, Lemma 2.2]\label{benri}
Let $p>2$ be a prime number and $B \subset \mathrm{GL}_2(\F_p)$ a Borel subgroup such that $B$ contains a matrix $g=\left(
 \begin{array}{cc}
 a&b\\
 0& c
 \end{array}
 \right)$
 with $a\neq c$. Let $B^{\prime}:=h^{-1}Bh$ with 
 $h=\left(
 \begin{array}{cc}
 1&b/(c-a)\\
 0& 1
 \end{array}
 \right)$. 
 \begin{itemize}
\item[$(1)$]
We can decompose $B^{\prime}$ as
\[
B^{\prime}=B^{\prime}_d\cdot B^{\prime}_1,
\]
and $B/[B, B] \cong B^{\prime}/[B^{\prime}, B^{\prime}]$. Here the groups $B^{\prime}_d, B^{\prime}_1$ are defined as
\[
B^{\prime}_d=B^{\prime} \cap 
\left\{
\left(
 \begin{array}{cc}
 a&0\\
 0& c
 \end{array}
 \right) \middle|\ a,c \in \F_p^{\times}
 \right\},\ \ 
B^{\prime}_1=B^{\prime} \cap 
\left\{
\left(
 \begin{array}{cc}
 1&b\\
 0&1
 \end{array}
 \right) \middle|\ b \in \F_p
 \right\}.
\]
\item[$(2)$]
$[B^{\prime}, B^{\prime}]=B^{\prime}_1$.
\end{itemize}
\end{prop}

\begin{proof}[Proof of Lemma {\rm \ref{nonsplitproperty}}]
We show that $g\in G_{\Q_p}$ acts on $M^0_f$ via a matrix
\begin{eqnarray}\label{sayou}
 \left(
 \begin{array}{cc}
 \omega_{\mathrm{cyc}}^{k-1}(g)\bar{\psi}^{-1}(g) &\bar{v}(g) \\
 0&\bar{\psi}(g)
 \end{array}
 \right),
\end{eqnarray}
where $\bar{v}(g) \in \F_p$ and $\bar{v}(G_{\Q^{\mathrm{ab}}_p})\neq 0$, choosing a suitable basis of $M^0_f$. Then this implies especially that $M_{f, D}=M^0_f(\omega_{\cyc}^{1-k/2}\chi_D)$ does not split as a $G_{\Q^{\mathrm{ab}}_p}$-module. We put $$\mathcal{D}_p:=\bar{\rho}^0_f(G_{\Q_p})\subset\mathrm{GL}_2(\F_p).$$
Since we assume $p-1\nmid k-1$, $\omega_{\cyc}^{k-1}$ is not a trivial character. Hence, if we take $g \in I_p$ such that $\omega_{\cyc}^{k-1}(g) \neq 1$, then $A:=\bar{\rho}^0_f(g) ( \in \mathcal{D}_p)$ is of the form $
A= \left(
 \begin{array}{cc}
 a &b \\
 0&1
 \end{array}
 \right)\ \ (a\in\F_p^{\times},\ b \in \F_p)
$
and $a \neq1$. Thus, $\mathcal{D}_p$ satisfies the assumptions in Proposition \ref{benri}. Then we change the basis $\{m_1, m_2\}$ of $M$ by the regular matrix $h=\left(
 \begin{array}{cc}
 1&b(1-a)^{-1}\\
 0& 1
 \end{array}
 \right)$ and write $\mathcal{D}^{\prime}_p = h^{-1}\mathcal{D}_ph$. Note that this change of basis affects only the upper right component of (\ref{actM}). From Proposition \ref{benri}, we have a decomposition of $\mathcal{D}^{\prime}_p$ as 
\begin{eqnarray*}\label{bunkai}
\mathcal{D}^{\prime}_p &=& (\mathcal{D}^{\prime}_p)_d \cdot(\mathcal{D}^{\prime}_p)_1\\
&=&
\left\{\left(
 \begin{array}{cc}
 \omega_{\cyc}^{k-1}(g)\bar{\psi}^{-1}(g)&0\nonumber\\
 0& \bar{\psi}(g)
 \end{array}
 \right) \middle|\ g \in G_{\Q_p} \right\}\cdot
 \left\{\left(
 \begin{array}{cc}
 1&b\\
 0&1
 \end{array}
 \right)\middle|\ b \in \F_p\right\}.
\end{eqnarray*}
Since we assume that $M^0_f$ does not split as a $G_{\Q_p}$-module, we have $(\mathcal{D}^{\prime}_p)_1\neq \{ I \}$. From Proposition \ref{benri} $(2)$, this means that the commutator subgroup of $\mathcal{D}^{\prime}_p$ is non-trivial which implies that $\mathcal{D}^{\prime}_p$ is non-abelian. Therefore, we have the desired result. 
This completes the proof of Proposition \ref{nonsplitproperty}.
\end{proof}

Thus, we get the conditions which implies $V^{G_{\Q_p}}=0$ and $M^{G_{\Q_p}}=(A[p])^{G_{\Q_p}}=0$ in Propositions \ref{compss}, \ref{compV} and \ref{comp}. The latter equality especially implies $A^{G_{\Q_p}}\otimes \F_p=0$. From the inequality (\ref{ineq}), we obtain
\[
\mathrm{dim}_{\F_p}(\mathrm{Im}(\Res^{\ur}_p))\leqslant \mathrm{dim}_{\F_p}(\mathrm{Im}(\mathrm{Loc}_p))\leqslant 1
\]
which is the claim of Step 3 under the assumptions of Propositions \ref{compss} and \ref{comp}. 
\subsection{CM case}
If $f$ is ordinary at $p$, $k=2$ and $T/p^nT$ splits for all $n$, we can show that the above inequality $\mathrm{dim}_{\F_p}(\mathrm{Im}(\Res^{\ur}_p))\leqslant 1$ still holds even when the situation (i) in Proposition \ref{comp} occurs.
\begin{prop}\label{cm}
Suppose that $k=2$. If $p\nmid D$, $f$ is ordinary at $p$, $T/p^nT$ splits as a $G_{\Q_p}$-module for all $n$ and $a_p(f)\equiv 1\pmod p$, then we have $$\mathrm{dim}_{\F_p}(\mathrm{Im}(\Res^{\ur}_p))\leqslant 1.$$
\end{prop}
\begin{proof}
When the situation (i) in Proposition \ref{comp} occurs, we know that $M^{G_{\Q_p}}$ is 1-dimensional over $\F_p$ from the proof of Proposition \ref{comp}. Then we can see that $A^{G_{\Q_p}}\otimes \F_p$ is also 1-dimensional. In fact, since $V$ does not contain a trivial representation of $G_{\Q_p}$, $A^{G_{\Q_p}}$ can not be $p$-divisible. While, $A^{G_{\Q_p}}$ does not contain any two linearly independent elements over $\Z_p$ by $\mathrm{dim}_{\F_p} M^{G_{\Q_p}}=1$ and the same argument as in the proof of Lemma \ref{suftam}. Hence, $A^{G_{\Q_p}}$ is isomorphic to $\Z/p^m \Z$ for some integer $m$ which implies that $\mathrm{dim}_{\F_p}(A^{G_{\Q_p}}\otimes \F_p)=1$. Thus, we have 
$$
\mathrm{dim}_{\F_p}(\mathrm{Im}(\mathrm{Loc}_p))\leqslant \mathrm{dim}_{\F_p}(H^1_f(\Q_p, M))=1+1=2
$$
from (\ref{ineq}). 

Note that in the situation (i), we assume $p\nmid D$ and hence $\chi_D$ is trivial when restricted to the inertia subgroup $I_p$. Thus, for every $n \in \Z_{\geqslant 1}$, $I_p$ acts on $T/p^nT$ via the diagonal matrix
\begin{eqnarray*}
 \left(
 \begin{array}{cc}
 \omega^{(n)}_{\cyc}(g) &0 \\
 0&1
 \end{array}
 \right),
\end{eqnarray*}
since $k=2$ and $T/p^nT$ splits as a $G_{\Q_p}$-module. Here, $\omega^{(n)}_{\cyc}$ denotes the mod $p^n$ cyclotomic character. Therefore, we obtain 
\[
(A[p^n])^{I_p} \cong (T/p^nT)^{I_p} \cong \Z/p^n\Z
\]
and 
\[
A^{I_p} \otimes \F_p \cong (\Q_p/\Z_p) \otimes \F_p =0.
\]
On the other hand, there is a commutative diagram
\[
  \xymatrix{
 0 \ar[r]^{}&  A^{G_{\Q_p}} \otimes \F_p\cong \Z/p\Z \ar[d]\ar[r]  & H^1_f(\Q_p, M) \ar[d]^{\Res_{\Q_p^{\ur}/\Q_p}} \\
 0 \ar[r]^{}&   A^{I_p} \otimes \F_p=0 \ar[r] &  H^1(\Q^{\ur}_p, M). 
  }
\]
Thus, the restriction map $\Res_{\Q_p^{\ur}/\Q_p}$ is not injective and its kernel has at least dimension 1 over $\F_p$. Since we have a decomposition $\Res^{\ur}_p=\Res_{\Q_p^{\ur}/\Q_p} \circ \mathrm{Loc}_p$, we get the desired inequality $$\mathrm{dim}_{\F_p}(\mathrm{Im}(\Res^{\ur}_p))\leqslant 2-1=1,$$
\end{proof}

\begin{rem}
Let $T$ be the integral $p$-adic Tate module of an elliptic curve $E$ over $\Q_p$, and suppose that $E$ has complex multiplication over an extension of $\Q_p$. Then $T/p^nT$ splits as a $G_{\Q_p}$-module for every $n$. Thus, the splitting condition in Proposition~\ref{cm} holds for $T$. It is conjectured by Ghate that for an ordinary modular form $f$, $T^0_f/p^n$ splits as a $G_{\Q_p}$-module for all $n$ if and only if $f$ has complex multiplication. For details, see \cite{Ga} for example.
\end{rem}

To summarize, we obtain the following proposition.
\begin{prop}\label{st3}
We assume the following{\rm :}
\begin{itemize}
\item If $f$ is supersingular at $p$, then $k\leqslant p+1$ holds.
\item If $f$ is ordinary at $p$, then $p-1\nmid k-1$ and neither of the conditions below occur:\vspace{2mm}

 \noindent$($When $p\nmid D)$
 \begin{itemize}
 \item[{\rm ($\mathrm{i}_1$)}] $k>2$, $M$ splits as a $G_{\Q_p}$-module, $p-1$ divides one of $k/2$ and $1-k/2$, and $a_p(f) \equiv 1 \pmod p$.
 \item[{\rm ($\mathrm{i}_2$)}] $k=2$, $M$ splits as a $G_{\Q_p}$-module and $T/p^nT$ does not split for some $n$, and $a_p(f) \equiv 1 \pmod p$.
 \item[{\rm (ii)}] $M$ does not split as a $G_{\Q_p}$-module, $p-1\mid k/2$ and $a_p(f) \equiv 1 \pmod p$.
  \end{itemize}
 \noindent$($When $D=p^{\ast}:=(-1)^{\frac{p-1}{2}}p.)$
 \begin{itemize}
 \item[{\rm (iii)}] $M$ splits as a $G_{\Q_p}$-module, $p-1$ divides one of $\frac{k-p+1}{2}$ and $\frac{k+p-3}{2}$, and $a_p(f) \equiv 1 \pmod p$.
 \item[{\rm (iv)}] $M$ does not split as a $G_{\Q_p}$-module, $p-1 \mid \frac{k-p+1}{2}$ and $a_p(f) \equiv 1 \pmod p$.
 \end{itemize} 
 \end{itemize}
 Then we have the following inequality{\rm :} $$\mathrm{dim}_{\F_p}(\mathrm{Im}(\Res^{\ur}_p))\leqslant 1.$$
\end{prop}
This completes the proof of Theorem \ref{main} as we explained in Subsection 2.3. 

\section{Numerical examples of Theorem \ref{main}}
We finally introduce some numerical examples of Theorem \ref{main}. To do this, we consider the situations where the inequality 
\begin{eqnarray}\label{gutaire}
\mathrm{dim}_{\F_p}\left( \Sha^{\BK}(\Q, A_{f, D})[p]\right) \geqslant 2
\end{eqnarray}
holds. This inequality implies the condition $\mathrm{dim}_{\F_p}( H_f^1(\Q, M_{f, D})) \geqslant 2$ in Theorem \ref{main} due to the existence of the surjection in (\ref{nume}). We know that the $\F_p$-dimension of $\Sha^{\BK}(\Q, A_{f, D})[p]$ is even because we have the isomorphism (\ref{selfdual}) and the generalized Cassels-Tate pairing for $A_{f, D}$ in \cite{Fl}. Hence, the inequality (\ref{gutaire}) holds if and only if $\Sha^{\BK}(\Q, A_{f, D})$ is just non-trivial. We study when this occurs assuming Bloch and Kato's conjecture.
 \subsection{Ratios of central $L$-values.}
In the following, we assume $N=1$ and the Hecke field $\Q(f)$ of $f$ is $\Q$. Recall that we put $z$ as a variable in the complex upper half plane and $q:=e^{2\pi\sqrt{-1}z}$. For a quadratic discriminant $D$ and a modular form $f(z)=\sum a_nq^n$, we consider the twisted $L$-function of $f$ by the quadratic character $\chi_D$
\[
L(f, \chi_D, s):=\sum_{n=1}^{\infty} \frac{\chi_D(n)a_n}{n^s}.
\]

We consider its central value $L(f, \chi_{D}, k/2)$. Suppose that $L(f, \chi_{D}, k/2) \neq 0$. Then Bloch and Kato's conjecture for $f\otimes \chi_D$ predicts such a value in terms of arithmetic invariants including the order of $\Sha^{\BK}(\Q, A_{f, D})$. The conjecture implies an equality of the $p$-adic valuations 
\begin{eqnarray}\label{bkformula}
v_p\left(\frac{L(f, \chi_{D}, k/2)}{\mathrm{vol}_{\infty}(\chi_{D}, 1-k/2)}\right)= v_p\left(\frac{\# \Sha^{\BK}(\Q, A_{f, D})\cdot c(\Q_p, A_{f, D})}{(\#\Gamma_{\Q}(A_{f, D}))^2}\right).
\end{eqnarray}
Here $\mathrm{vol}_{\infty}(\chi_{D}, 1-k/2)$ denotes a certain transcendental part of the value $L(f, \chi_D, k/2)$, $c(\Q_p, A_{f, D})$ the $p$-part of the Tamagawa factor of $A_{f, D}$ at $p$ and
$\Gamma_{\Q}(A_{f, D}):=H^0(\Q, A_{f, D})$. For the precise definition of $\mathrm{vol}_{\infty}(\chi_{D}, 1-k/2)$, see \cite[Section2]{DSW} for example. We omit the definition of $c(\Q_p, A_{f, D})$ here. Note that the product of Tamagawa factors outside $p$ does not appear above, since we assume $N=1$. Suppose that $f$ satisfies the conditions in Theorem \ref{main}. Then $M_{f, D}=A_{f, D}[p]$ is irreducible as a $G_{\Q}$-module which implies $\#\Gamma_{\Q}(A_{f, D})=1$. Thus, (\ref{bkformula}) yields
\begin{eqnarray}\label{bkformula2}
v_p\left(\frac{L(f, \chi_{D}, k/2)}{\mathrm{vol}_{\infty}(\chi_{D}, 1-k/2)}\right)= v_p\left(\# \Sha^{\BK}(\Q, A_{f, D})\cdot c(\Q_p, A_{f, D})\right).
\end{eqnarray}

For the transcendental factor $\mathrm{vol}_{\infty}(\chi_{D}, 1-k/2)$, we have the following property.
\begin{lem}[\cite{Du}, Lemma 6.1]\label{inf}
For $D > 0$, we have $$\displaystyle \mathrm{vol}_{\infty}(\chi_{D}, 1-k/2)=\frac{\mathrm{vol}_{\infty}(1-k/2)}{\sqrt{D}}.$$ Here $\mathrm{vol}_{\infty}(1-k/2)$ denotes the transcendental part of the L-value $L(f, \mathbf{1}, k/2)$ for the trivial character $\mathbf{1}$.
\end{lem}

Suppose that we have two positive quadratic discriminants $D$ and $D^{\prime}$ such that neither $L(f, \chi_{D}, k/2)$ nor $L(f, \chi_{D^{\prime}}, k/2)$ vanishes, and the condition
\begin{eqnarray}\label{tamdif}
c(\Q_p, A_{f, D}) = c(\Q_p, A_{f, D^{\prime}})
\end{eqnarray}
is satisfied. A sufficient condition for (\ref{tamdif}) is given in \cite[Lemma 6.3]{Du}, for example. We consider the ratio of $L(f, \chi_{D}, k/2)$ and $L(f, \chi_{D^{\prime}}, k/2)$. Using Lemma \ref{inf} and (\ref{bkformula2}) for $D$ and $D^{\prime}$ together with (\ref{tamdif}), we have 
\begin{eqnarray}\label{L-value}
v_p \left( \frac{\sqrt{D}}{\sqrt{D^{\prime}}}\cdot\frac{L(f, \chi_{D}, k/2)}{L(f, \chi_{D^{\prime}}, k/2)}\right) = v_p\left(\frac{\# \Sha^{\BK}(\Q, A_{f, D})}{\# \Sha^{\BK}(\Q, A_{f, D^{\prime}})}\right).
\end{eqnarray}

On the other hand, the twisted $L$-value $L(f, \chi_{D}, k/2)$ was studied by Kohnen and Zagier in \cite{KZ} in terms of Shimura's theory of modular forms of half integral weight. Shimura's theory gives a correspondence between modular forms of half integral weight and modular forms of even integral weight. We write $S_{k}(\mathrm{SL}_2 (\Z))$ as the space of cusp forms of even weight $k$ on the full modular group $\mathrm{SL}_2 (\Z)$ and $S_{\frac{k+1}{2}}(\Gamma_0(4))$ the space of cusp forms of weight $\frac{k+1}{2}$ on the congruence subgroup $\Gamma_0(4)$. In \cite{K}, Kohnen defined a certain subspace $S^{+}_{\frac{k+1}{2}}(\Gamma_0(4))$ of $S_{\frac{k+1}{2}}(\Gamma_0(4))$ and showed that Shimura's correspondence induces an isomorphism 
\[
\kappa: S^{+}_{\frac{k+1}{2}}(\Gamma_0(4)) \xrightarrow{\sim} S_{k}(\mathrm{SL}_2 (\Z)).
\]
 In \cite[Theorem 1]{KZ}, Kohnen and Zagier gave a formula of the value $L(f, \chi_{D}, k/2)$ for $f \in S_{k}(\mathrm{SL}_2 (\mathbb{Z}))$ in terms of the $|D|$-th Fourier coefficient of $\kappa^{-1}(f)$.
\begin{thm*}[Kohnen-Zagier]
Let $f \in S_{k}(\mathrm{SL}_2 (\Z))$ be a normalized Hecke eigenform and  $g:=\kappa^{-1}(f)=\sum^{\infty}_{n=1} c_n q^n\in S^{+}_{\frac{k+1}{2}}(\Gamma_0(4))$ the inverse image of $f$ under $\kappa$. Let $D$ be a quadratic discriminant with $(-1)^{k/2}D>0$. Then
\[
\frac{c^2_{|D|}}{\langle g, g\rangle}= \frac{(k/2-1)!}{\pi^{k/2}} |D|^{(k-1)/2} \frac{L(f, \chi_D, k/2)}{\langle f, f\rangle},
\]
where $\langle \cdot , \cdot\rangle$ denotes the Petersson inner product.
\end{thm*}

Now, for example we put 
$$f=\Delta(z):=\sum^{\infty}_{n=1} \tau_n q^n 
\quad  (q=e^{2\pi i z}),
$$
Ramanujan's cusp form of weight $k=12$ and level $N=1$. We take two positive quadratic discriminants $D$ and $D^{\prime}$ such that $L(\Delta, \chi_{D}, 6), L(\Delta, \chi_{D^{\prime}}, 6)\neq 0$, and they satisfy (\ref{tamdif}). Using the theorem of Kohnen and Zagier, we obtain an equality
\[
\frac{\sqrt{D}}{\sqrt{D^{\prime}}}\cdot\frac{L(\Delta, \chi_{D}, 6)}{L(\Delta, \chi_{D^{\prime}}, 6)} = \frac{c^2_D}{c^2_{D^{\prime}}}\cdot \left({\frac{D^{\prime}}{D}}\right)^5,
\]
where $c_{D}, c_{D^{\prime}}$ are the $D$-th and $D^{\prime}$-th Fourier coefficients of $\kappa^{-1}(\Delta(z))$ respectively. From (\ref{L-value}), we obtain an equality of $p$-adic valuations
\begin{eqnarray}\label{Sha}
v_p\left(\frac{\# \Sha^{\BK}(\Q, A_{\Delta, D})}{\# \Sha^{\BK}(\Q, A_{\Delta, D^{\prime}})}\right)= v_p\left(\frac{c^2_D}{c^2_{D^{\prime}}}\cdot \left({\frac{D^{\prime}}{D}}\right)^5\right).
\end{eqnarray}
Therefore, we can deduce $\Sha^{\BK}(\Q, A_{\Delta, D})\neq 0$ under the Bloch-Kato conjecture for $\Delta\otimes\chi_{D}$ and $\Delta\otimes\chi_{D^{\prime}}$ if there exists a pair $(D, D^{\prime})$ such that the right-hand side of (\ref{Sha}) is positive. 

On the other hand, the main example in \cite{KZ} provides a formula 
\begin{eqnarray}\label{kz}
\displaystyle \kappa^{-1}(\Delta(z))=\frac{60}{2\pi i }\left(2G_4 (4z)\theta^{\prime}(z)-G^{\prime}_4(4z)\theta(z)\right),
\end{eqnarray}
 where $$G_4(z):=\frac{1}{240}+\sum^{\infty}_{n=1} \sigma_3(n) q^n \quad (\sigma_3(n):=\sum_{0<d|n} d^3),$$
 $$\theta(z):=1+2\sum^{\infty}_{n=1} q^{n^2}.$$
 Note that $\theta^{\prime}(z)$ and $G^{\prime}_4(4z)$ denote the derivatives with respect to $z$. Thus, we can compute the Fourier coefficients of $\kappa^{-1}(\Delta(z))$ with (\ref{kz}) explicitly and hence the right-hand side of (\ref{Sha}).

\subsection{Examples}
In this subsection, we assume the Bloch-Kato conjecture and give two application examples of Theorem \ref{main}. 
\begin{prop}\label{exam1}
We set $p=11$, $f(z)=\Delta(z)$, $D=517$ and $D^{\prime}=33$, and assume the Bloch-Kato conjecture for $\Delta \otimes \chi_{517}$ and $\Delta \otimes \chi_{33}$. Then there exists a surjective $\Gal(K_{\Delta, 517}/\Q)$-equivariant homomorphism from $\cl(K_{\Delta, 517})\otimes\F_{11} $ to $M_{\Delta, 517}$.
\end{prop}
\begin{proof}
First, we check the conditions $(A)$, $(B)$, $(C)$ and $(D)$ in Theorem \ref{main} hold for this setting. Since $N=1$, $(A)$ and $(D)$ are automatically satisfied. The prime number $p=11$ is a good ordinary prime of $\Delta(z)$, and $D=517$ and $D^{\prime}=33$ are proper multiples of $11$. Therefore, $(B)$ holds. It is known that the image of $\bar{\rho}^0_{\Delta}$ contains $\mathrm{SL}_2(\F_p)$ except for the cases $p=2, 3, 5, 7, 23$ and $691$, hence $(C)$ holds. 

We can also check that the condition (\ref{tamdif}) is satisfied for $D=517$ and $D^{\prime}=33$ by \cite[Lemma 6.3]{Du}.

 Next, we show $\Sha^{\BK}(\Q, A_{\Delta, 517})\neq 0$ under the Bloch-Kato conjecture. By (\ref{kz}), we compute the $11i$-th Fourier coefficient of $\kappa^{-1}(\Delta(z))=\sum^{\infty}_{n=1} c_n q^n$ for $2\leqslant i \leqslant 47$ as follows:
\begin{center}
\begin{longtable}[h]{|c|c||c|}
\hline
$i$ & $11i$ & $c_{11i}$ \\
\hline
\endfirsthead
\multicolumn{3}{c}%
{\tablename\ \thetable\ -- \textit{Continued from previous page}} \\
\hline
$i$ & $11i$ & $c_{11i}$ \\
\hline
\endhead
\hline \multicolumn{3}{c}{\textit{Continued on next page}} \\
\endfoot
\hline
\endlastfoot
     $2$ & $22$ & 0    \\ \hline
        $3$ & $\mathbf{33}$  & $-6480=-2^4 \cdot 3^4 \cdot5$     \\ \hline
         $4$ & $44$ &  $-43680=-2^5\cdot 3 \cdot 5 \cdot 7 \cdot 13$   \\ \hline
    $5$& $55$ & $0$ \\ \hline
     $6$ & $66$& $0$ \\ \hline
     $7$ & $77$& $110880=2^5\cdot 3^2\cdot 5 \cdot 7 \cdot 11$ \\ \hline
     $8$ & $88$& $-153120=2^5 \cdot 3 \cdot 5 \cdot 11 \cdot 29$ \\ \hline
     $\vdots$ & $\vdots$& $\vdots$ \\ \hline
          $40$ & $440$& $-25903680=-2^6\cdot 3 \cdot 5 \cdot 11 \cdot 223$ \\ \hline
     $41$ & $451$& $0$ \\ \hline
     $42$ & $462$& $0$ \\ \hline
     $43$ & $473$& $-6850800=-2^4 \cdot 3^2 \cdot 5^2 \cdot 11 \cdot 173$ \\ \hline
     $44$ & $484$& $-20919416=-2^3\cdot 7 \cdot 373561$ \\ \hline
     $45$ & $495$& $0$ \\ \hline
     $46$ & $506$& $0$ \\ \hline
      $47$ & $\mathbf{517}$& $52000080=2^4\cdot 3\cdot 5 \cdot  \mathbf{11} \cdot 19697$ \\ \hline

     \end{longtable} 
\end{center}\vspace{-8mm}
Thus, using (\ref{Sha}), we obtain 
\begin{eqnarray*}
v_{11}\left(\frac{\# \Sha^{\BK}(\Q, A_{\Delta, 517})}{\# \Sha^{\BK}(\Q, A_{\Delta, 33})}\right)=v_{11}\left(\frac{c^2_{517}}{c^2_{33}}\cdot \left({\frac{3}{47}}\right)^5\right)=v_{11}\left(\frac{52000080^2\cdot 3^5}{6480^2\cdot 47^5}\right)>0.
\end{eqnarray*}
Hence, $\Sha^{\BK}(\Q, A_{\Delta, 517})\neq 0$ which implies $\mathrm{dim}_{\F_{11}}(\Sha^{\BK}(\Q, A_{\Delta, 517})[11]) \geqslant 2$.  Applying Theorem \ref{main}, we can see that there exists a surjective $\Gal(K_{\Delta, 517}/\Q)$-equivariant homomorphism $\cl(K_{\Delta, 517})\otimes\F_{11} \twoheadrightarrow M_{\Delta, 517}$. 
\end{proof}

\begin{rem}
 We note that the representation $M_{\Delta, 517}$ in Proposition \ref{exam1} comes from an elliptic curve over $\mathbb{Q}$. Put $E$ as the modular curve $X_0(11)$ and $f_E = \sum^{\infty}_{n=1} a_n q^n \in S_2(\Gamma_0(11))$ the corresponding rational cusp form. For $\Delta(z)=\sum^{\infty}_{n=1} \tau_n q^n$ and $f_E$, we have the congruences of coefficients $\tau_n \equiv a_n \pmod{11}$ which induces an isomorphism between Galois representations $M^0_{\Delta}$ and $E[11]$. Thus, we have an isomorphism $M_{\Delta, 517} \cong E_{517}[11]\otimes \omega_{\cyc}^{5}$ where $E_{517}$ denotes the quadratic twist of $E$ by $517$. We note here that $E_{517}$ has bad reduction at $11$. Therefore, we can not treat this representation $M_{\Delta, 517} \cong E_{517}[11]\otimes \omega_{\cyc}^{5}$ by the previous work of Prasad and Shekhar \cite{PS} since their result deals with an elliptic curve and its good prime $p$.
\end{rem} 
\begin{prop}\label{exam2}
We set $p=67$, $f(z)=\Delta(z)$, $D=2881$ and $D^{\prime}=201$, and assume the Bloch-Kato conjecture for $\Delta\otimes \chi_{2881}$ and $\Delta \otimes \chi_{201}$. Then there exists a surjective $\Gal(K_{\Delta, 2881}/\Q)$-equivariant homomorphism from $\cl(K_{\Delta, 2881})\otimes\F_{67}$ to $M_{\Delta, 2881}$.
\end{prop}
\begin{proof}
First, we check the conditions $(A)$, $(B)$, $(C)$ and $(D)$ in Theorem \ref{main} hold for this setting. Since $N=1$, $(A)$ and $(D)$ hold. The prime number $p=67$ is a good ordinary prime of $\Delta(z)$, and $D=2881$ and $D^{\prime}=201$ are proper multiples of $67$. Hence, we can see that $(B)$ holds. As we noted in the proof of Proposition \ref{exam1}, the condition $(C)$ also holds.

Since $p=67>k=12$, we can also check that $c(\Q_{67}, A_{\Delta, 2881}) = c(\Q_{67}, A_{\Delta, 201})=1$ by \cite[Lemma 4.6]{DSW}, and the condition (\ref{tamdif}) is satisfied for $D=2881$ and $D^{\prime}=201$.

Next, we show $\Sha^{\BK}(\Q, A_{\Delta, 2881})\neq 0$ under the Bloch-Kato conjecture. For $i \in \mathbb{Z}$ with $2\leqslant i \leqslant 43$, the $67i$-th Fourier coefficient of $\kappa^{-1}(\Delta(z))=\sum^{\infty}_{n=1} c_n q^n$ is computed as follows:

\begin{center}
\begin{longtable}{|c|c||c|}
\hline
$i$ & $67i$ & $c_{67i}$ \\
\hline
\endfirsthead
\multicolumn{3}{c}
{\tablename\ \thetable\ -- \textit{Continued from previous page}} \\
\hline
$i$ & $67i$ & $c_{67i}$ \\
\hline
\endhead
\hline \multicolumn{3}{c}{\textit{Continued on next page}} \\
\endfoot
\hline
\endlastfoot
     $2$ & $134$ & 0    \\ \hline
        $3$ & $\mathbf{201}$  & $-2686320=-2^4\cdot 3^2 \cdot 5 \cdot 7 \cdot 13 \cdot 41$     \\ \hline
         $4$ & $268$  &  $-4016160=-2^5\cdot 3^2\cdot 5\cdot 2789$   \\ \hline
    $5$ & $335$ & 0 \\ \hline
    $6$ & $402$ & 0 \\ \hline
     $7$ & $469$ & $-32215680=2^7 \cdot 3^2 \cdot 5 \cdot 7 \cdot 17 \cdot 47$ \\ \hline
      $8$ & $536$& $24612000=2^5 \cdot 3 \cdot 5^3 \cdot 7 \cdot 293$ \\ \hline
       $\vdots$ & $\vdots$ & \vdots \\ \hline
         $36$& $2412$ & $-36145440=-2^5 \cdot 3^4 \cdot 5 \cdot 2789$ \\ \hline
    $37$ & $2479$& 0 \\ \hline
    $38$ & $2546$ & 0 \\ \hline
    $39$ & $2613$& $-981246240=-2^5 \cdot 3^2 \cdot 5 \cdot 13 \cdot 23 \cdot 43 \cdot 53$ \\ \hline
     $40$ &$2680$& $3359129280=2^6 \cdot 3 \cdot 5 \cdot 13 \cdot 17 \cdot 71 \cdot 223$   \\ \hline
    $41$&$2747$ & 0   \\ \hline
     $42$&$2814$ & 0    \\ \hline
        $43$&$\mathbf{2881}$  & $-2622438960=-2^4 \cdot 3 \cdot 5 \cdot \mathbf{67} \cdot 71 \cdot 2297$     \\ \hline
    \end{longtable}
\end{center}\vspace{-5mm}

Using (\ref{Sha}) for $D=2881$ and $D^{\prime}=201$, we obtain
\begin{eqnarray*}\label{67}
v_{67} \left(\frac{\# \Sha^{\BK}(\Q, A_{\Delta, 2881})}{\# \Sha^{\BK}(\Q, A_{\Delta, 201})}\right)=v_{67}\left( \frac{c^2_{2881}}{c^2_{201}}\cdot \left({\frac{3}{43}}\right)^5\right)=v_{67}\left(\frac{2622438960^2\cdot 3^5}{2686320^2\cdot 43^5}\ \right)>0.
\end{eqnarray*}
Thus, we have $\Sha^{\BK}(\mathbb{Q}, A_{\Delta, 2881})\neq 0$ to obtain 
$$\mathrm{dim}_{\mathbb{F}_{67}}(\Sha^{\BK}(\mathbb{Q}, A_{\Delta, 2881})[67]) \geqslant 2.$$
Hence, Theorem \ref{main} implies the existence of a surjective $\Gal(K_{\Delta, 2881}/\Q)$-equivariant homomorphism $\cl(K_{\Delta, 2881})\otimes\F_{67} \twoheadrightarrow M_{\Delta, 2881}$.
\end{proof}

Moreover, we can show the following property for this example.
\begin{prop}\label{konai}
The Galois representation $M_{\Delta, 2881}$ in Proposition {\rm \ref{exam2}} does not come from any elliptic curve over $\Q$. In other words, there does not exist any elliptic curve $E$ over $\Q$ such that 
\[
M_{\Delta, 2881}\simeq E[67]\otimes \omega_{\cyc}^i\ \ \ (0\leqslant i\leqslant 65),
\]
where $\omega_{\cyc}$ denotes the mod $67$ cyclotomic character.
\end{prop}
\begin{proof}
It suffices to show that there does not exist any elliptic curve $E$ over $\Q$ such that
\[
M^0_{\Delta}\simeq E[67]\otimes \omega_{\cyc}^i\ \ \ (0\leqslant i\leqslant 65).
\]

We assume the existence of such an elliptic curve $E$. Since $\Delta(z)$ is good ordinary at $67$, an element $g$ in the inertia subgroup $I_{67}$ at 67 acts on $M^0_{\Delta}$ as 
\[
 \left(
 \begin{array}{cc}
 \omega_{\cyc}^{11}(g) &\bar{u}(g) \\
 0&1
 \end{array}
 \right),
\]
where $\bar{u}(g)$ is as in (\ref{actM}). Putting $F(M^0_{\Delta})$ as a field which corresponds to the kernel of $\bar{\rho}^0_{\Delta}\mid_{I_{67}}$, we can see that the semi-simplification $(M_{\Delta}^0)^{\mathrm{ss}}$ of $M^0_{\Delta}$ as an $\F_{67}[\Gal(F(M^0_{\Delta})/\Q_{67}^{\ur})]$-module is 
\begin{eqnarray}\label{ss1}
\F_{67}(\omega_{\cyc}^{11})\oplus \F_{67}.
\end{eqnarray}

First, suppose that $E$ has good reduction at $67$. Since $M_{\Delta}^0$ is reducible when restricted to $G_{\Q_{67}}$, $E$ has good ordinary reduction at $67$. Hence, $g \in I_{67}$ acts on $E[67]\otimes \omega_{\cyc}^i$ as 
\[
 \left(
 \begin{array}{cc}
 \omega_{\mathrm{cyc}}(g) &\bar{v}(g) \\
 0&1
 \end{array}
 \right)\cdot \omega_{\cyc}^{i},
\]
where $\bar{v}(g) \in \mathbb{F}_{67}$. Then we have another description of 
$(M_{\Delta}^0)^{\mathrm{ss}}$ as 
\begin{eqnarray}\label{ss2}
\F_{67}(\omega_{\cyc}^{i+1})\oplus\F_{67}(\omega_{\cyc}^i).
\end{eqnarray}
However, (\ref{ss1}) never coincides with (\ref{ss2}) for any $i$. This is a contradiction.

Next, we suppose that $E$ has bad reduction at $67$. When $E$ has multiplicative reduction at $67$, the theory of the Tate curve says that $g \in I_{67}$ acts on $E[67]$ via  
$
 \left(
 \begin{array}{cc}
  \omega_{\mathrm{cyc}}(g) &\bar{w}(g) \\
 0&1
 \end{array}
 \right),
$
where $\bar{w}(g) \in \F_{67}$. Hence, we get the same conclusion as in the case where $E$ has good reduction at $67$. When $E$ has additive potentially multiplicative reduction, it is known that $E$ acquires split multiplicative reduction over a ramified quadratic extension $L$ over $\Q_{67}^{\ur}$. Thus, we can also treat this case as the good reduction case, studying the semi-simplification of $M_{\Delta}^0$ as an $\F_{67}[\Gal(L\cdot F(M^0_{\Delta})/L)]$-module.

Finally, we assume that $E$ has additive potentially good reduction at $67$. It is a well-known fact that $E$ acquires good reduction over a totally ramified extension $L^{\prime}$ of degree 4 or 6 over $\Q_{67}^{\ur}$. See \cite{Kr} for more details. Thus, we can deduce a contradiction as the good reduction case, studying the semi-simplification of $M_{\Delta}^0$ as an $\F_{67}[\Gal(L^{\prime}\cdot F(M^0_{\Delta})/L^{\prime})]$-module.
\end{proof}

\begin{rem}
In the above Examples \ref{exam1} and \ref{exam2}, it seems difficult to compute the ideal class groups $\cl(K_{\Delta, 517})$ and $\cl(K_{\Delta, 2881})$ without assuming the Bloch-Kato conjecture. For example, it is difficult to calculate them by machine computation since the extension degree $[K_{\Delta, D} : \Q]$ is huge in general. In fact, the Galois group $\Gal(K_{\Delta, D}/\Q)$ contains $\mathrm{SL}_2(\F_p)$ due to the assumption $(C)$ in Theorem \ref{main}. Hence, $[K_{\Delta, D} : \Q]$ is greater than or equal to $\#\mathrm{SL}_2(\F_p)=(p-1)p(p+1)$. 
\end{rem}

\subsection*{Acknowledgement}
The author would like to thank his supervisor Masato Kurihara heartily for his continued support, guiding the author to the topic in this paper and helpful discussions. He would also like to thank Dipendra Prasad and Sudhanshu Shekhar for valuable comments on his previous paper \cite{D}. Thanks are also due to Ryotaro Sakamoto for providing a number of valuable comments. He is also grateful to Neil Dummigan for introducing the paper \cite{Du} kindly to him. Finally, he thanks deeply the anonymous referees for careful reading of his manuscript. This research was supported by JSPS KAKENHI Grant Number 21J13502.


\begin{thebibliography}{99}
 \bibitem{BK} S. Bloch and K. Kato, $L$-functions and Tamagawa numbers of motives, The Grothendieck Festschrift Volume I, Progress in Mathematics. 86, Birkh\"auser, Boston (1990), 333--400.
\bibitem{EC} J. Bosman, Computations with Modular Forms and Galois Representations, In Computational Aspects of Modular Forms and Galois Representations : How One Can Compute in Polynomial Time the Value of Ramanujan's Tau at a Prime (AM-176), edited by Bas Edixhoven and Jean-Marc Couveignes, 129-157. Princeton University Press (2011). 
  \bibitem{D} N. Dainobu, Ideal class groups of number fields and Bloch-Kato's Tate-Shafarevich groups for symmetric powers of elliptic curves. to appear in Tokyo J. Math.
  \bibitem{De} P. Deligne, Formes modulaires et repr\'esentations $\ell$-adiques, S\'eminaire Bourbaki 1968/69, Exp. 355,
Lect. Notes in Math. \textbf{179}, Springer, Berlin, Heidelberg, New York (1971), 139--172.
\bibitem{Du} N. Dummigan, Congruences of modular forms and Selmer groups. Mathematical Research Letters 
\textbf{8} (2001), 479--494.
\bibitem{DSW} N. Dummigan, W. Stein and M. Watkins, Constructing elements in Shafarevich-Tate groups of modular motives, Number Theory and Algebraic Geometry (2004), 91--118.
\bibitem{Edi} B. Edixhoven, The weight in Serre's conjectures on modular forms. Invent Math \textbf{109} (1992), 563--594.
\bibitem{Ei} M. Eichler, Quaternare quadratische Formen und die Riemannische Vermutung f\"ur die Kongruenzzeta-funktion, Arch. Math. \textbf{5} (1954), 355--366.
\bibitem{Fl} M. Flach, A generalisation of the Cassels-Tate pairing, J. reine angew. Math. \textbf{412} (1990), 113--127.
\bibitem{Ga} E. Ghate, On the local behavior of ordinary modular Galois representations, Modular curves and abelian varieties, Progress in Mathematics vol. 224, Birkh\"auser (2004), 105-124.
\bibitem{He} J. Herbrand, Sur les classes des corps circulaires, Journal de Math\'ematiques Pures et Appliqu\'ees 
\textbf{11} (1932), 417--441.
\bibitem{H} T. Hiranouchi, Local torsion primes and the class numbers associated to an elliptic curve over $\mathbb{Q}$, Hiroshima Math. J. \textbf{49}, no. 1 (2019), 117--127.
\bibitem{K}  W. Kohnen, Modular forms of half-integral weight on $\Gamma_0 (4)$. Math. Ann. \textbf{248} (1980), 249--266.
\bibitem{KZ} W. Kohnen and D. Zagier, Values of $L$-series of modular forms at the center of the critical strip, Invent. Math. \textbf{64} (1981), 173--198.
\bibitem{Kr} A. Kraus, Sur le d\'eaut de semi-stabilit\'e des courbes elliptiques \`a r\'e-duction additive, Manuscripta Math. \textbf{69} (1990), 353--385.
\bibitem{LR} A. Lozano-Robledo, Division fields of elliptic curves with minimal ramification, Revista Matematica Iberoamericana, \textbf{31}, no. 4 (2015), 1311--1332.
\bibitem{BA} B. Mazur and A. Wiles, Class fields of abelian extensions of $\Q$, Invent. Math. \textbf{76} (1984), 179--330.
\bibitem{Oh} T. Ohshita, Asymptotic lower bound of class numbers along a Galois representation, J. Number
Theory \textbf{211} (2020), 95--112.
\bibitem{PS} D. Prasad and S. Shekhar, Relating the Tate-Shafarevich group of an elliptic curve with the class group, Pacific J. Math. \textbf{312}, no.1 (2021), 203--218.
\bibitem{Ri} K. Ribet, A modular construction of unramified $p$-extensions of $\Q(\mu_p)$, Invent. Math. {\bf 34} (1976), 151--162.
\bibitem{SY} F. Sairaiji and T. Yamauchi, On the class numbers of the fields of the $p$-torsion points of certain
elliptic curves over $\Q$, J. Number Theory \textbf{156} (2015), 277--289.
\bibitem{Sh} G. Shimura, Correspondances modulaires et les fonctions $\zeta$ des courbes alg\'ebriques, J. Math. Soc. Japan \textbf{10} (1958), 1--28.
\bibitem{CE} C. Skinner and E. Urban, The Iwasawa main
conjectures for $\mathrm{GL}_2$, Invent. Math. \textbf{195} (2014), 1--277.
\end{thebibliography}
\end{document}